\documentclass[a4paper,10pt]{amsart}
\usepackage{amssymb, amsmath, mathrsfs}
\usepackage{mathtools}
\mathtoolsset{showonlyrefs}
\usepackage[inline, final]{showlabels}
\usepackage{tikz-cd}
\usepackage{mathrsfs}
\usepackage{enumitem}
\usepackage{shuffle}
\usepackage{pdfpages}
\usepackage{subcaption}
\usepackage{float}

\usepackage[]{todonotes}
\usepackage[
 pdfauthor={Nils Matthes  and Morten S. Risager},
 pdftitle={The distribution of Manin's iterated integrals of modular forms},
 pdfsubject={Primary 11F67; Secondary  11F72, 11M36},
 pdfkeywords={},
 pdfproducer={LaTeX2e},
 pdfcreator={Pdflatex},
 pdfdisplaydoctitle =true,
 psdextra]
 {hyperref}
\usepackage{pdfsync}

\newtheorem{thm}{Theorem}[section]
\newtheorem{prop}[thm]{Proposition}

\newtheorem{cor}[thm]{Corollary}
\newtheorem*{rem*}{Remark}

\theoremstyle{definition}

\newtheorem{example}[thm]{Example}

\newtheorem{rem}{Remark}

\def \e {{\varepsilon}}
\def \a {{\mathfrak{a}}}
\def \b {{\mathfrak{b}}}
\def \R {{\mathbb R}}
\def \H {{\mathbb H}}
\def \E {{\mathbb E}}
\def \N {{\mathbb N}}

\def \C {{\mathbb C}}
\def \g {\gamma}
\def \G {\Gamma}
\def \Z {\mathbb{Z}}
\def \psl  {{\hbox{PSL}_2( {\mathbb R})} }
\def \sl  {{\hbox{SL}_2( {\mathbb R})} }

\def \slz  {{\hbox{SL}_2( {\mathbb Z})} }
\def \GmodH {{\Gamma\backslash\mathbb H}}

\DeclareMathOperator{\sgn}{sgn}

\newcommand{\vol}[1]{\hbox{vol}\left( #1 \right)}
\newcommand{\abs}[1]{\left\lvert #1 \right\rvert}
\newcommand{\norm}[1]{\left\lVert #1 \right\rVert}

\newcommand{\inprod}[2]{\left \langle #1,#2 \right\rangle}

\title{The distribution of Manin's iterated integrals of modular forms.}
\date{today}
\author[N.  Matthes]{Nils Matthes}
\address{Department of Mathematical
  Sciences, University of Copenhagen, Universitets\-parken 5, 2100
  Copenhagen \O , Denmark}
\email{nils.oliver.matthes@gmail.com}

\author[M. Risager]{Morten S. Risager}
\address{Department of Mathematical
  Sciences, University of Copenhagen, Universitets\-parken 5, 2100
  Copenhagen \O , Denmark}
\email{risager@math.ku.dk}

\thanks{The research of Nils Matthes was supported by a Walter Benjamin Fellowship of the DFG.The research of  Morten S. Risager was supported by the Grant DFF-7014-00060B from Independent Research Fund Denmark}

\keywords{}
\subjclass[2020]{Primary 11F67; Secondary 11F72, 11M36  }
\date{\today}

\begin{document}
\begin{abstract} We determine the asymptotic distribution of Manin's iterated integrals of length  at most 2. For all lengths we compute all the asymptotic moments. We show that if the length is at least 3 these moments do in general \emph{not} determine a unique distribution.
\end{abstract}

\maketitle
\section{Introduction}

In a series of papers Chen \cite{Chen:1967, Chen:1971, Chen:1973a, Chen:1977} introduced the very influential concept of iterated integrals of differential forms. 
Later Manin \cite{Manin:2005, Manin:2006}  studied and further developed these in 
the case of  holomorphic forms on hyperbolic surfaces. Hain \cite[Sec. 5]{Hain:1998} had previously introduced the notion of iterated integrals with values in a flat bundle. Manin's iterated integrals generalise both the 
classical theory of modular symbols (\cite[Sec. 3]{Manin:2009}) and certain aspects of the theory of multiple zeta-values (see \cite{GilFresan:2017} and the references therein,  
\cite{ChoieIhara:2013, BruggemanChoie:2016, Choie:2016, ChoieIhara:2013, ChoieMatsumoto:2016, DeitmarHorozov:2013}).  The theory has fascinating connections to physics \cite{BognerBrown:2015}, 
mixed Tate motives \cite{DeligneGoncharov:2005,Brown:2012,Brown:2019c}, 
deformation quantization \cite{Kontsevich:1999, BanksPanzerPym:2020a}, 
knot invariants \cite{LeMurakami:1995, LeMurakami:1996,IharaTakamuki:2001}, and many other areas. 

Manin's iterated integrals (also called noncommutative modular symbols) are defined as follows: Let $f_1\ldots f_l$ be holomorphic cusp forms of weight 2 for the congruence group $\Gamma_0(q)$. Let 
\begin{equation}
I_{i\infty}^{\gamma i\infty}(f_1, \ldots f_l)\coloneqq \int_{i\infty}^{\g i \infty}f_1(z_1)\left(\int_0^{z_1}f_2(z_2)\ldots\left(\int_0^{z_{l-1}}f_l(z_l)dz_l\right)\ldots dz_2\right)dz_1.
\end{equation}

In this paper we study the statistical properties of these numbers when $\g\in \G_0(q)$, or -- what amounts to the same thing -- when $a/c=\g i\infty$ runs through the cusps of $\Gamma_0(q)$ equivalent to the cusp at infinity. To this end we consider \begin{equation}T(M)=\{a/c\in \mathbb Q\cap [0,1), c\leq M, (a,c)=1, q\vert c\},\end{equation} which we use to order the iterated integrals. We study the properties of the random variable 
\begin{equation}
Z_M(A):=\frac{\#\{\frac{a}{c}\in T(M)\vert\, \left(\frac{C_q}{\log(c)}\right)^{l/2} I_{i\infty}^{a/c}(f_1,\ldots f_l)\in A\}}{\#T(M)},\end{equation}
 in particular what happens when $M\to\infty.$ Here $C_q=\pi\cdot [\slz:\Gamma_0(q)]/24$, and  $A\subseteq\mathbb C$ is a Borel set.

If $l=1$ then $I_{i\infty}^{\gamma i\infty}(f_1)$ is a classical modular symbol. Such modular symbols play a prominent role in the study of modular forms (see e.g. \cite{Birch:1971,Manin:1972a, Mazur:1973a, Cremona:1997a}). This paper was motivated by the following result: 

\begin{thm}\label{length1}
If $l=1$ and $\norm{f}=1$ then $Z_M$, converges in distribution as $M\to\infty$ to the standard  complex normal distribution with distribution function $\frac{1}{\pi}e^{-\abs{z}^2}$.
\end{thm}
This was conjectured by Mazur and Rubin \cite{MazurRubin:2021}, proved by Petridis--Risager\cite{PetridisRisager:2018a} and reproved and extended by Nordentoft \cite{Nordentoft:2021c}, Lee--Sun \cite{LeeSun:2019}, Constantinescu \cite{Constantinescu:2020a}, Bettin--Drappeau \cite{BettinDrappeau:2019}, and Nordentoft--Drappeau \cite{DrappeauNordentoft:2022a}. We also give yet another proof of this result.

In this paper we investigates what happens when $l>1$. Write $f=(f_1,\ldots f_l)$.

\begin{thm}\label{intro-thm:final-convergence-length-two}
 If $l=2$ then $Z_M$ converges in distribution as $M\to\infty$ to a radially symmetric distribution $X(f)$. The distribution $X(f)$  depends only on the Gram matrix $\{\inprod{f_i}{f_j}\}_{i,j=1}^2$ related to the Petersson weight 2 inner product.  
\end{thm} 

The distribution in Theorem  \ref{intro-thm:final-convergence-length-two} can be given concretely in certain cases: \begin{enumerate}[label=\alph*)]
    \item If $f_1=f_2$ with $\norm{f_i}=1$ then $X(f)$ is the  Kotz-like distribution (\cite{Nadarajah:2003}) with distribution function $\frac{1}{\pi\abs{z}}e^{-2\abs{z}}$. 
    \item If $f_1,f_2$ form an orthonormal set then $X(f)$ has distribution function \begin{equation}
        \frac{1}{4}\int_{0}^1\frac{1}{y(1-y)}\frac{\displaystyle\sinh\left(\frac{\pi \abs{z}}{2\sqrt{y(1-y)}}\right)}{\cosh^2\left(\displaystyle\frac{\pi \abs{z}}{2\sqrt{y(1-y)}}\right)}dy. 
    \end{equation} Proving this explicit form involves a non-trivial combinatorial identity due to Fran\c con and Viennot \cite{FranconViennot:1979a}. See Example \ref{fun-example} for details.
\end{enumerate}

We now turn to the length $l\geq 3$ case: 
\begin{thm}
 \label{intro-thm:final-convergence-length-three} If $l\geq 3$ then  all asymptotic moments of $Z_M$ exist as $M\to\infty$  and there exists at least one rotationally invariant distribution with these as its moments. 
\end{thm}

We emphasize that Theorem  \ref{intro-thm:final-convergence-length-three} does not directly allow us to conclude convergence in distribution. In fact, in general the possible limit distributions are \emph{indetermined}, i.e there are infinitely many distributions whose moments agree with the asymptotic moments of $Z_M$. 

\begin{rem}
We can say quite a bit more when $f_1=f_2=\cdots =f_l$. In this case  \begin{equation}I_{a}^b(f)=\frac{I_a^b(f_1)^l}{l!},\end{equation} see \eqref{powers}. This allows us to conclude, using Theorem \ref{length1} and general probability theory results (more precisely the Portmanteau theorem), that $Z_M$ converges in distribution to $\frac{Z^l}{l!}$ where $Z$ is the standard complex normal distribution. A computation shows that this equals the Kotz-like distribution with distribution function \begin{equation}f\textsubscript{Kotz,$l$}(z)=\frac{(l!)^2}{l\pi }\frac{e^{-(l!\abs{z})^{2/l}}}{(l!\abs{z})^{2(1-l^{-1})}}.\end{equation} This Kotz-like distribution is known to be indetermined when $l\geq 3$, and hence it is not possible to conclude the convergence in distribution of $Z_M$ simply by knowing its moments. 
\end{rem} 

\begin{rem}It is possible to express the distribution results explained above as a distribution result of central values of additive twists of multiple $L$-functions. These additive twists of multiple $L$-functions are $GL_2$ generalizations of multiple polylogs which themselves generalize multiple zeta values (See e.g. \cite{GilFresan:2017}).

Define, for $\Re(s)\gg 1$, the multiple $L$-functions 
\begin{equation}
L(f,s)=\sum_{0=m_{l+1}<m_l<\cdots m_1}\frac{\prod_{i=1}^la_i(m_i-m_{i+1})}{m_l\cdots m_2 m_1^s}
\end{equation}
where $f=(f_1,f_2, \ldots,f_l)$. Here 
\begin{equation}f_i(z)=\sum_{n=1}^\infty a_i(n)q^n, \quad q=e(z)=\exp(2\pi i z)\end{equation}
is the Fourier expansion of $f_i$. The function $L(f,s)$ converges absolutely for $\Re(s)\gg 1$, admits holomorphic continuation to $s\in\C$, and satisfies a functional equation 
 relating $s$ and $2-s$ (see Theorem \ref{thm:functional-eq-multiple-zeta}). Given a rational number $a/c$ with $q\vert c$ we define the additive twist of $L(f,s)$ as  follows: Define, for $\Re(s)\gg 1$ 
\begin{equation}
L(f,\frac{a}{c},s)=\sum_{0=m_{l+1}<m_l<\cdots m_1}\frac{\prod_{i=1}^la_i(m_i-m_{i+1})}{m_l\cdots m_2 m_1^s}e(\frac{a}{c}m_1).
\end{equation}

Then $L(f,\frac{a}{c},s)$ converges absolutely for $\Re(s)\gg 1$, admits analytic continuation to  $s\in\C$ and satisfies a functional equation 
relating $s$ and $2-s$ (see Theorem \ref{thm:functional-eq-multiple-zeta-twisted}). The central value for this function satisfies \begin{equation}
    L(f,a/c, 1)=(2\pi i)^lI_{i\infty}^{a/c}(f_1\ldots f_l),
\end{equation}
and therefore Theorems
\ref{length1}, \ref{intro-thm:final-convergence-length-two}, and \ref{intro-thm:final-convergence-length-three}
gives a distribution result about the statistics of these central values. 
\end{rem}

 \begin{rem}Theorems \ref{intro-thm:final-convergence-length-two} and \ref{intro-thm:final-convergence-length-three} are proved  using the method of moments. In doing so the shuffle product formula for iterated integrals (see \eqref{shuffle-product}) plays a crucial role as an elegant device for handling higher moments.  An even more crucial tool in this paper is a higher order Eisenstein series twisted by Manin's iterated integrals, which we denote by $E^{v,w}(z,s)$ (See \eqref{def:higher-order-twist} for the definition). Such a series was first defined by Chinta, Horozov and O'Sullivan in \cite{ChintaHorozovOSullivan:2019} generalizing further an Eisenstein series twisted by modular symbols introduced by Goldfeld \cite{Goldfeld:1999a, Goldfeld:1999b}. Goldfeld's Eisenstein series has been studied intensively in \cite{OSullivan:2000a, Petridis:2002a, PetridisRisager:2004a, BruggemanDiamantis:2016a, JorgensonOSullivan:2008a, PetridisRisager:2018a}.
Chinta, Horozov and O'Sullivan proved several interesting properties about this series. We give new proofs of (some of) these results. Going further we find precise information about the pole at $s=1$ for $E^{v,w}(z,s)$ and its zero Fourier coefficient (See Theorems \ref{E-properties} and  \ref{thm:dirichlet-series-prop}). This is what ultimately allows us to conclude Theorems \ref{intro-thm:final-convergence-length-two} and \ref{intro-thm:final-convergence-length-three}. Our techniques also gives a new proof of Theorem \ref{length1}.
\end{rem}

\begin{rem}
There are several possible generalizations of the results in this paper: general weight, Eisenstein series, Bianchi groups, Maass forms, etc. For the length 1 case this has been done in \cite{Nordentoft:2021c}, \cite{Bettin:2019},  \cite{Constantinescu:2020a}, \cite{DrappeauNordentoft:2022a}. 

\end{rem}

\begin{rem}
The paper is organized as follows: In Section \ref{sec:Manin-integrals} we review Manin's construction of interated integrals and some of its properties. In Section \ref{sec:Eisenstein-twisted} we first define the twisted Eisenstein series $E^{v,w}(z,s)$ as well as some of its related series. We then use the spectral theory of the automorphic Laplacian to find the analytic properties. In Section \ref{sec:Moments-computation} we use these analytic properties combined with the shuffle relations for the iterated integrals to compute all the moments for two random variables; $X^v_M$ related to the Eisenstein series, and $Y^v_M$ related to its zero Fourier coefficient. 
In Section \ref{sec:probability} we prove a Fr\'echet-Shohat-type theorem for complex random variables which allows us to conclude Theorem \ref{length1} \ref{intro-thm:final-convergence-length-two}, \ref{intro-thm:final-convergence-length-three} in Section \ref{sec:distribution}. In Section \ref{sec:L-series} we analyse the relation between Manin's iterated integrals and additively twisted multiple $L$-series. Finally in Section \ref{sec:numerics} we show some numerics for $I_{i\infty}^{a/c}(f)$. 
\end{rem}

\section{Manin's noncommutative modular symbols}\label{sec:Manin-integrals}

The purpose of this section is to review the definition, as well as some properties, of Manin's noncommutative modular symbols \cite{Manin:2006}.

Let $S$ be a connected Riemann surface. We fix a finite set $V\subseteq \Omega^1(S)$ of holomorphic differential one-forms on $S$, and consider the free monoid $V^\ast$ on $V$. By definition, an element of $V^\ast$ is a word $v=v_1\ldots v_n$, where $v_i\in V$ --- this includes the empty word $\varepsilon$. We let $\ell(v)$ denote the length of the word $v$. Given $v=v_1\ldots v_n \in V^\ast$ and a piecewise smooth path $\gamma: [0,1]\rightarrow S$, we define their iterated integral by
\begin{equation} \label{eqn:iteratedintegral}
I_\gamma(v)\coloneqq \int_0^1f_1(t_1)\left(\int_0^{t_1}f_2(t_2)\ldots\left(\int_0^{t_{n-1}}f_n(t_n)dt_n\right)\ldots dt_2\right)dt_1,
\end{equation}
where $f_i(t)dt\coloneqq \gamma^*v_i$. If $v=\varepsilon$, then $I_\gamma(v)\coloneqq 1$. Also, since $dv_i=v_j\wedge v_k=0$ for all $v_i,v_j,v_k \in V$, an application of Stokes' theorem shows that the value of $I_\gamma(v)$ only depends on the homotopy class of $\gamma$ relative to its endpoints.

The case of interest for us is when $S=\mathbb H^\ast$ is the extended upper half-plane with its canonical coordinate $z$, and $V$ consists of one-forms of the shape $v=f(z)dz$, where $f(z)$ is a holomorphic cusp form of weight two for some subgroup cofinite subgroup $\Gamma \subseteq \operatorname{SL}_2(\mathbb R)$. As $\mathbb H^\ast$ is simply connected, the iterated integral \eqref{eqn:iteratedintegral} only depends on the endpoints of $\gamma$, so that we may unambiguously write $I_{\gamma(0)}^{\gamma(1)}(v)$ instead of $I_\gamma(v)$. Writing $\g(t_i)=z_i$ we see that 
\begin{equation}
    I_{a}^b(v)=\int_{a}^bf_1(z_1)\int_a^{z_1} f_2(z_2)\cdots \int_a^{z_{n-1}}f_n(z_n)dz_ndz_{n-1}\cdots dz_1
\end{equation}

In the special case where $\gamma(0)$ and $\gamma(1)$ are both cusps, the iterated integral is an example of a \emph{noncommutative modular symbol}. Also, note that for any fixed $a \in \mathbb H^\ast$, the function $z\mapsto I_a^z(v)$ is holomorphic.

In order to succinctly express properties of iterated integrals, it is customary to consider their generating series. To this end, let $X_V\coloneqq \{X_v \, : \, v\in V\}$ be a set of noncommuting variables indexed by $V$. We extend the notation $X_v$ to the case of a word $v=v_1\ldots v_n\in V^\ast$ by setting $X_v\coloneqq \prod_{i=1}^n X_{v_i}$. Now, denoting by $\Omega_V$ the formal differential one-form
\begin{equation}
\Omega_V\coloneqq \sum_{v\in V}v\cdot X_v,
\end{equation}
we define, for $a,b\in \mathbb H^\ast$, the series
\begin{equation}
J_a^b(\Omega_V)\coloneqq \sum_{v\in V^\ast}I_a^b(v)\cdot X_v \in \mathbb C\langle\!\langle X_V\rangle\!\rangle,
\end{equation}
where $\mathbb C\langle\!\langle X_V\rangle\!\rangle$ denotes the $\mathbb C$-algebra of formal power series in the variables $X_v$, for $v\in V$, with multiplication given by concatenating words.
\begin{prop} \label{prop:symbols-properties}
The series $J_a^b(\Omega_V)$ has the following properties.
\begin{enumerate}[label={(\roman*})]
\item \label{derivative}\emph{(Differential equation)}
For fixed $a \in \mathbb H^\ast$, the function $z\mapsto J_a^z(\Omega_V)$ satisfies
\begin{equation}
dJ_a^z(\Omega_V)=\Omega_V\cdot J_a^z(\Omega_V).
\end{equation}
\item \label{decomposition-paths}
\emph{(Composition of paths formula)}
We have
\begin{equation}
J_b^c(\Omega_V)\cdot J_a^b(\Omega_V)=J_a^c(\Omega_V),
\end{equation}
for all $a,b,c\in \mathbb H^\ast$
\item\label{gamma-invariance}
\emph{($\Gamma$-invariance)}
For all $\gamma\in \Gamma$, we have
\begin{equation}
J_{\gamma a}^{\gamma b}(\Omega_V)=J_a^b(\Omega_V).
\end{equation}
\item
\emph{(Grouplike property)} \label{shuffle}
Let $\Delta': \mathbb C\langle\!\langle X_V\rangle\!\rangle\rightarrow \mathbb C\langle\!\langle X_V\rangle\!\rangle  \, \widehat{\otimes} \, \mathbb C\langle\!\langle X_V\rangle\!\rangle$ be the continuous $\mathbb C$-algebra homomorphism given by $\Delta'(X_v)=X_v\otimes 1+1\otimes X_v$, for all $v\in V$. Then
\begin{equation}
\Delta'(J_a^b(\Omega_V))=J_a^b(\Omega_V) \, \widehat{\otimes} \, J_a^b(\Omega_V).
\end{equation}
\end{enumerate}
\end{prop}
The grouplike property of $J_a^b(\Omega_V)$ gives an explicit formula for the product of any two of its coefficients. To elaborate on this, consider the free $\mathbb Q$-vector space $\mathbb Q\langle V\rangle$ on $V^\ast$, and define the \emph{shuffle product} by
\begin{equation}
\begin{aligned}
\shuffle: \mathbb Q\langle V\rangle\otimes \mathbb Q\langle V\rangle &\longrightarrow \mathbb Q\langle V\rangle\\
v_1\ldots v_m\otimes v_{m+1}\ldots v_{m+n}&\mapsto \sum_{\sigma\in \Sigma_{m,n}}v_{\sigma(1)}\ldots v_{\sigma(m+n)},
\end{aligned}
\end{equation}
where $\Sigma_{m,n}$ denotes the subset of all permutations $\sigma$ of the $\{1,\ldots,m+n\}$ such that $\sigma^{-1}$ is monotonously increasing on both $\{1,\ldots,m\}$, and on $\{m+1,\ldots,m+n\}$. By the duality between the shuffle product and the coproduct $\Delta'$ \cite[\S 1.5]{Reutenauer:1993}, Proposition \ref{prop:symbols-properties}.(iv) is then equivalent to the equality
\begin{equation}\label{shuffle-product}
I_a^b(v\shuffle w)=I_a^b(v)I_a^b(w),
\end{equation}
for all words $v,w\in V^\ast$. Of particular interest for us is the case of powers
\begin{equation}\label{shuffle-product-formula}
I_a^b(v)^n=\sum_{w\in V^\ast}c_{v^{\shuffle n}}(w)I_a^b(w),
\end{equation}
for an integer $n\geq 0$, 
where $c_{v^{\shuffle n}}(u) \in \mathbb Z\geq 0$ denotes the coefficient of the word $u \in V^\ast$ in the $n$-fold shuffle product $v\shuffle \ldots \shuffle v$. Note that for $v=v_1\cdots v_l\in V^\ast$
\begin{align}
\label{eq:sum-coefficients}
\sum_{u\in V^\ast}c_{v^{\shuffle n}}(u)&=\frac{(\ell(v)n)!}{(\ell(v)!)^n},
\intertext{and}
\label{eq:lower-bounds-coefficients}
\sum_{u\in V^\ast}c_{v^{\shuffle n}}(u)&\geq c_{v^{\shuffle n}}(v_1^n\cdots v_l^n)=(n!)^{\ell(v)}.
\end{align}
Note also that if $v_1\in V$ then \eqref{shuffle-product-formula} gives that \begin{equation}\label{powers}
    I_a^b(v_1)^n=n! I_a^b(v_1^n).
\end{equation}

\section{Eisenstein series twisted with noncommutative modular symbols}\label{sec:Eisenstein-twisted}
Chinta, Horozov, and O'Sullivan defined in \cite{ChintaHorozovOSullivan:2019} an Eisenstein series twisted by Manin's noncommutative modular symbols. This series is a  generalization of a certain Eisenstein series twisted by modular symbols introduced by Goldfeld \cite{Goldfeld:1999a, Goldfeld:1999b}. In this section we introduce (a slight variant of) the Eisenstein series introduced in \cite{ChintaHorozovOSullivan:2019} and prove several new results about it. 

Let $\H$ denote the upper half-plane equipped with the hyperbolic metric $ds$ and corresponding measure $d\mu(z)=y^{-2}dxdy$. The group of orientation preserving isometries is isomorphic to $\psl$ via linear fractional transformations, and the action of $\psl$ extends to the boundary of $\H$. Let $k$ be an even integer and let, for   $z\in \H$,
\begin{equation*}
    j_\gamma(z)=\frac{j(\g,z)}{\abs{j(\g,z)}} 
\end{equation*} where $j(\g,z)=cz+d$ where $c,d$ are the lower entries in $\gamma\in \sl.$ Note that \begin{equation}\label{j-properties}j(\g_1\g_2,z)=j(\g_1,\g_2 z)j(\g_2,z).\end{equation} Let $k$ be an even integer. For functions $f:\H\to \C$ we define for $\g\in \sl$   \begin{equation}
    (T_{k, \g}f)(z)=j^{-k}_\g(z) f(\g z).
\end{equation}
We have  
\begin{equation}\label{eq:compositions-slash}T_{k,\g_1}\circ T_{k,\g_2}=T_{k,\g_2\g_1}\end{equation} and since $k$ is even 
these maps factor through $\psl$.

Consider a cofinite, non-cocompact discrete subgroup  $\G$ of $\sl$. We assume for convenience that $-I\in \G$. Fix a set of inequivalent cusps of $\G$, and for each such cusp $\a$ fix a scaling matrix $\sigma_\a$; i.e a  matrix $\sigma_\a\in \sl$ satisfying $\sigma_\a(\infty)=\a$ and  $\G_\a=\sigma_\a\Gamma_\infty\sigma_{\a}^{-1}$. Here $\Gamma_\a$ is the stabilizer of $\a$ in $\Gamma$ and $\Gamma_\infty$ is the standard parabolic subgroup generated by the matrices corresponding to $z\mapsto z+1$.

Consider now two finite sets  of holomorphic cuspidal 1-forms on $\GmodH$
\begin{align}
    V&=\{\omega_1, \ldots ,\omega_n\}\\
    W&=\{\omega'_1, \ldots ,\omega'_m\},
\end{align} with corresponding total differential 
\begin{align}
    \Omega_V(z)&=\sum_{v\in V}v(z)X_v\\
    \Omega_W(z)&=\sum_{w\in W}w(z)Y_w
\end{align}
Here $X=\{X_w\vert w\in V\}$, $Y=\{Y_w\vert w\in W\}$ are free formal noncommuting (real) variables, where $X_vY_w=Y_wX_v$ for $v\in V, w\in W$. As before we also write $X_w=X_{w_1}\cdots X_{w_l}$ , for $w=w_1\cdots w_l\in V^*$, and similar with $Y_w$. 

Let $a\in \H^*$, $k$ an even integer, and $\a$ a cusp for $\G$ with scaling matrix $\sigma_\a$. 
We then define the \emph{total higher order twisted Eisensteins series} to be
\begin{equation}
\label{E-total}
E^{**}_\a{(z,s,k,\Omega_V, \Omega_W)}:=\sum_{\g\in \G_\a\backslash \Gamma} j_{\sigma_{\a}^{-1}\g}^{-k}(z)J_a^{\g a}(\Omega_V)\overline{J_a^{\g a}(\Omega_W)}\Im(\sigma_{\a}^{-1}\g z)^s,
\end{equation}
when $\Re(s)>1$. It follows from \cite[Cor 2.6]{ChintaHorozovOSullivan:2019} that $E^{**}_\a{(z,s,k,\Omega_V, \Omega_W)}$  converges absolutely and uniformly for $\Re(s)\geq 1+\e$, for any $\e>0$. We emphasize  that this means \emph{precisely} that the coefficients, the  \emph{higher order twisted Eisensteins series}, 
\begin{equation}
\label{def:higher-order-twist}
E^{vw}_\a{(z,s,k)}:=\sum_{\g\in \G_\a\backslash \Gamma} j_{\sigma_{\a}^{-1}\g}^{-k}(z)I_a^{\g a}(v)\overline{I_a^{\g a}(w)}\Im(\sigma_{\a}^{-1}\g)^s,    
\end{equation} in the formal expansion
\begin{equation}
E^{**}_\a{(z,s,k,\Omega_V, \Omega_W)}=\sum_{(v,w)\in V^{*}\times W^{*}}E^{v,w}_\a(z,s,k)X_vY_w,
\end{equation}
converge absolutely and uniformly for $\Re(s)\geq 1+\e$. These series are so-called higher order form. We refer to \cite{ChintaHorozovOSullivan:2019, JorgensonOSullivan:2008a, ChintaDiamantisOSullivan:2002a, ChoieDiamantis:2006, KlebanZagier:2003a, DiamantisSim:2008} for various results about higher order forms. 

\begin{rem}\label{rem:comparison}
To compare our notation with that of Chinta, Horozov, and O'Sullivan \cite[(3.3)]{ChintaHorozovOSullivan:2019} one may show, using the reversal of path formula \cite[Thm 3.19]{GilFresan:2017}
that our series   $E^{**}_\a{(z,s,0,\Omega_V, \Omega_W)}$ up to notational differences equals 
$\mathcal E_\a{(z,s)}$ when ${\bf{f}}=(-f_{\omega_1}, \ldots,-f_{\omega_n})$ and ${\bf{g}}=(-f_{\omega'_1}, \ldots,-f_{\omega'_m})$
 as defined in \cite[(3.3)]{ChintaHorozovOSullivan:2019}. 
\end{rem}

A complication arising from the fact that $E^{**}_\a{(z,s,0,\Omega_V, \Omega_W)}$ is a higher order form, and in particular not automorphic, is that it is not so easy to analyze using the spectral theory of the automorphic Laplacian. For this reason we introduce a related automorphic function as follows:

For $k$ an even integer and $\a$ a cusp for $\G$ with scaling matrix $\sigma_\a$ as in section \ref{sec:maass-background}  we define the \emph{total automorphic twisted Eisensteins series} 
\begin{equation}
\label{D-total}
D^{**}_\a{(z,s,k,\Omega_V, \Omega_W)}:=\sum_{\g\in \G_\a\backslash \Gamma}T_{k, \sigma_{\a}^{-1}\g} (J_\a^{\sigma_\a z}(\Omega_V)\overline{J_\a^{\sigma_\a z}(\Omega_W)}y^s),
\end{equation}
when $\Re(s)>1$. This defines also implicitly the coefficient functions, the \emph{automorphic twisted Eisenstein series},  
\begin{equation}D^{v,w}_\a(z,s,k):=\sum_{\g\in \G_\a\backslash \Gamma}T_{k, \sigma_{\a}^{-1}\g} (I_\a^{\sigma_\a z}(v)\overline{I_\a^{\sigma_\a z}(w)}y^s).\end{equation}
It follows again from \cite[Cor 2.6]{ChintaHorozovOSullivan:2019} that $D^{**}_\a{(z,s,k,\Omega_V, \Omega_W)}$  converges absolutely and uniformly for $\Re(s)\geq 1+\e$, for any $\e>0$. 

It is obvious from the definition and \eqref{eq:compositions-slash} that $D^{**}_\a{(z,s,k,\Omega_V, \Omega_W)}$ and $D^{v,w}_\a(z,s,k)$ are automorphic of weight $k$ for $\G$. Note that in the sum defining $D^{v,w}_\a(z,s,k)$ the contribution coming from the $\g=I$ equals $I_\a^{z}(v)\overline{I_\a^{z}(w)}y^s$ which for $\Re(s)>1$ goes to zero as $y\to\infty$. It follows from well-known properties of $E_\a{(z,s,0)}$ that $D^{v,w}_\a(z,s,k)$ is square integrable, so \begin{equation}\label{square-integrable}D^{v,w}_\a(z,s,k)\in L^2(\G, k)\end{equation} when $(v,w)\neq {(\e,\e)}$, and $\Re(s)>1$. 
\begin{rem}\label{rem:back-forth}
We now describe the precise relation between the \lq higher order\rq{} Eisenstein series $E^{**}_\a{(z,s,k,\Omega_V, \Omega_W)}$ and the automorphic form $D^{**}_\a{(z,s,k,\Omega_V, \Omega_W)}$:  By Proposition \ref{prop:symbols-properties} we have, for $\g\in \G$ 
\begin{equation}
    J_a^{\g a}(\Omega_V)=J_a^\a(\Omega_V)J_\a^{\gamma z}(\Omega_V)J_{z}^a(\Omega_V)
\end{equation}
where we have used that $J_{\g z}^{\g a}(\Omega_V)=J_{z}^a(\Omega_V)$. 
It follows that 
\begin{align}\label{ref:back-forth}
\begin{split}
E^{**}_\a&{(z,s,k,\Omega_V, \Omega_W)}\\
&= J_{\a}^a(\Omega_V)\overline{J_{\a}^a(\Omega_W)}D^{**}_\a{(z,s,k,\Omega_V, \Omega_W)} J_{z}^a(\Omega_V)\overline{J_{z}^a(\Omega_W)}.
\end{split}
\end{align}
From this we see that 
\begin{equation}
    E^{vw}_\a{(z,s,k)}=\sum_{\substack{v'v''v'''=v\\
    w'w''w'''=w}}I_a^\a(v')\overline{I_a^\a(w')}D^{v''w''}_\a(z,s,k)I_z^a(v''')\overline {I_z^a(v''')} .
\end{equation}
\end{rem}

It follows from the work of Chinta, Horozov, and O'Sullivan \cite[Sec. 4]{ChintaHorozovOSullivan:2019} and a small induction argument that $D^{v,w}_\a(z,s,k)$ admits meromorphic contination to $s\in \C$. We will reprove this fact in Section \ref{sec:continuation-etc} using a different proof, which gives better information about the pole structure at $s=1$. 

\subsection{The Laplacian}\label{sec:maass-background}
In order to analyze the functions $D^{vw}_\a(z,s,k)$ introduced above we apply the spectral theory of the automorphic Laplacian. In this subsection we fix notation and review some classic results about Maass forms and the Laplacian which will be instrumental in achieving this. For a more thorough introduction the reader should consult the classics \cite{Maass:1949a, Roelcke:1966a, Selberg:1956, Selberg:1965a, Fay:1977a}, or some more recent sources like \cite{Bump:1997a, Iwaniec:2002a, DukeFriedlanderIwaniec:2002a, Michel:2007}. 

We consider the weight $k$ raising and lowering 
operators
\begin{equation}
R_k=(z-\overline{z})\frac{\partial}{\partial z} +\frac{k}{2},\quad  L_k=(z-\overline{z})\frac{\partial}{\partial \overline z} +\frac{k}{2}  
    \end{equation} where $\frac{\partial}{\partial z}=\frac{1}{2}(\frac{\partial}{\partial x}-i\frac{\partial}{\partial y})$.  Beware that there are slight variations in the literature in the definition of these. We have the generalized Leibniz rules 
\begin{align}\label{ref:generalized-Leibniz}
\begin{split}
    R_{k+l}(fg)&=(R_kf)g+fR_lf,\\ 
    L_{k+l}(fg)&=(L_kf)g+fL_lf,
\end{split}
\end{align}
and the relations
\begin{align}
\label{eq:RT-commutator}   
\begin{split}
R_{k}\circ T_{k,\g}&= T_{k+2,\g}\circ R_{k},\\
L_{k}\circ T_{k,\g}&= T_{k-2,\g}\circ L_{k}.\\
\end{split}
\end{align}
The weight $k$ Laplacian may be defined by 
\begin{equation}\Delta_k=-R_{k-2}L_k-\frac{k}{2}(1-\frac{k}{2})=y^2\left(\frac{\partial^2}{\partial x^2}+\frac{\partial^2}{\partial x^2}\right)-iky\frac{\partial}{\partial x}\end{equation}

Recall that a function on $\H$ is called automorphic of weight $k$ for the group $\G$ if $T_{k,\g}f=f$ for all $\g\in \G$. An important class of automorphic functions of weight $k$ are the non-holomorphic Eisenstein series by the convergent series
\begin{equation}
    E_\a{(z,s,k)}:=\sum_{\g\in \G_\a\backslash \Gamma}T_{k, \sigma_{\a}^{-1}\g} (y^s), \quad \textrm{ when }\Re(s)>1.
\end{equation}
Selberg proved that this admits meromorphic continuation to $s\in \C$. For $\Re(s)>1/2$ it has finitely many poles; they are all real, and $s=1$ is a pole if and only if $k=0$ and in this case the residue of the Eisenstein series at $s=1$ equals the reciprocal of $\vol{\GmodH}$, the volume of $\GmodH$. The Eisenstein series satisies
\begin{align}
\label{ref:Eisenstein-properties}
\begin{split}
    (\Delta_k+s&(1-s))E_\a(z,s,k)=0\\
    R_k E_\a(z,s,k)&=(\frac{k}{2}+s)E_\a(z,s,k+2)\\
    L_k E_\a(z,s,k)&=(\frac{k}{2}-s)E_\a(z,s,k-2)
\end{split}
\end{align}

Another important class of automorphic functions of weight $k$ are coming from holomorphic cusp forms of weight $k$. If $f\in S_k(\G)$ is such a form then $g(z)=y^{k/2}f(z)$ is automorphic of weight $k$, satisfying \begin{equation}\Delta_k g=\frac{\abs{k}}{2}(1-\frac{\abs{k}}{2})g,\end{equation}  and for $k>0$                                    
\begin{equation}\label{ref:kill}
    L_kg=0,\textrm{ and }R_{-k}\overline g=0.
\end{equation}

A smooth automorphic functions $f$ of weight $k$ for $\G$ growing at most polynomially at the cusps and satisfying $(\Delta_k+\lambda_f)f=0$ for some $\lambda_f\in \C$ is called a Maass form. Above we have seen two types of examples of such forms.

Consider the set $L^2(\G,k)$ of automorphic functions on $\H$ of weight $k$ which are square integrable with respect to the inner product 
\begin{equation}
    \inprod{f}{g}=\int_\GmodH f(z)\overline{g(z)}d\mu(z)
\end{equation}
The operator $\Delta_k$ induces a selfadjoint unbounded operator, called the automorphic Laplacian on $L^2(\G,k)$ which, by a standard abuse of notation, is also  denoted by $\Delta_k$. For $g_k$, $h_{k+2}$ smooth and  automorphic of weight $k$, $k+2$ respectively  we have 
\begin{equation}\label{ref:adjoint}
    \inprod{R_k g_k}{h_{k+2}}=\inprod{g_k}{L_{k+2}h_{k+2}},
\end{equation}
assuming that the decay properties of all the involved functions are such that the integrals converge absolutely.

For $s$ away from the spectrum of the weight $k$ automorphic Laplacian the weight $k$ resolvent operator  $R(k,s)$ (not to be confused with the raising operator $R_k$) is a bounded operator on $L^2(\G,k)$ satisfying 
\begin{equation}
R(s,k)=-(\Delta_k+s(1-s))^{-1}
\end{equation}
At isolated points of the spectrum this operator has simple poles, with residue essentially the projection to the corresponding eigenspace. In particular when $k=0$ and $s=1$ we have 
\begin{equation}\label{eq:resolvent-splitting}R(s,0)=\frac{P_0}{(s-1)}+R_0(s,0)\end{equation} where $P_0$ is the projection to the set of constants, and  $R_0(s,0)$ is holomorphic at $s=1$. The norm of the resolvent $R(s,k)$ is bounded by the reciprocal of the distance between $-s(1-s)$ and the spectrum of $\Delta_k$.  

\subsection{Taking derivatives}
We are now ready to see how the raising, lowering and Laplace operator acts on $D^{**}_\a{(z,s,k,\Omega_V, \Omega_W)}$. Since, in this computation $V$, $W$ are fixed we omit $\Omega_V$ and $\Omega_W$ from the notation for the twisted Eisenstein series and the total iterated integral.

We write $\omega_i=f_i(z)dz$, and $\omega'_i=g_i(z)dz$ where $f_i$, $g_i$ are holomorphic cusp forms of weight 2, and denote 
\begin{align}
    \Omega_{V,1}&=\sum_{i=1}^nf_iX_{\omega_i}\in S_2(\G)\langle\!\langle X\rangle\!\rangle, \\ \Omega_{W,1}&=\sum_{i=1}^m g_iY_{\omega'_i}\in S_2(\G)\langle\!\langle Y\rangle\!\rangle
\end{align}
Note that whereas $\Omega_{V}$ is a 1-form $\Omega_{V,1}$ is a function

\begin{prop}{\quad }\label{prop:LaplaceD*}
\begin{enumerate}[label=(\roman*)]
\setlength\itemsep{0.8em}
    \item \label{raising}
        $R_kD^{**}_\a(z,s,k)=(z-\overline z)\Omega_{V,1}D^{**}_\a(z,s,k)+(k/2+s)D^{**}_\a(z,s,k+2)$
    \item \label{lowering} 
    $ L_kD^{**}_\a(z,s,k)=(z-\overline z)\overline{\Omega_{W,1}}D^{**}_\a(z,s,k)+(k/2-s)D^{**}_\a(z,s,k-2)$
\item \label{laplace}
$\begin{aligned}[t]
        (\Delta_k+s(1-s))&D^{**}_\a(z,s,k)\\=&-(z-\overline z)^2\Omega_{V,1}\overline{\Omega_{W,1}}D^{**}_\a(z,s,k)\\
        & 
        -(k/2+s)(z-\overline z)\overline{\Omega_{W,1}}D^{**}_\a(z,s,k+2)\\& -(k/2-s)(z-\overline z)\Omega_{V,1}D^{**}_\a(z,s,k-2)
    \end{aligned}$
\item \label{laplace-better}
$\begin{aligned}[t]
        (\Delta_k+s(1-s))&D^{**}_\a(z,s,k)\\=&(z-\overline z)^2\Omega_{V,1}\overline{\Omega_{W,1}}D^{**}_\a(z,s,k)\\
        & 
        -(z-\overline z)(\overline{\Omega_{W,1}}R_k +\Omega_{V,1}L_k)D^{**}_\a(z,s,k)
    \end{aligned}$
\end{enumerate}

\end{prop}
\begin{proof}
We first use the generalized Leibniz rule \eqref{ref:generalized-Leibniz} and Proposition \ref{prop:symbols-properties},  \ref{derivative} to find
\begin{align}
    R_k(J_\a^{\sigma_{a}z}\overline{J_\a^{\sigma_{a}z}}y^s)&=    R_0(J_\a^{\sigma_{a}z}\overline{J_\a^{\sigma_{a}z}})y^s+    J_\a^{\sigma_{a}z}\overline{J_\a^{\sigma_{a}z}}R_ky^s\\
    &=R_0(J_\a^{\sigma_{a}z})\overline{J_\a^{\sigma_{a}z}}y^s +J_\a^{\sigma_{a}z}\overline{J_\a^{\sigma_{a}z}}(k/2+s)y^s\\
        &=(z-\overline z)\frac{\Omega_{V,1}(\sigma_\a z)}{j(\sigma_\a,z)^2 }(J_\a^{\sigma_{a}z})\overline{J_\a^{\sigma_{a}z}}y^s +J_\a^{\sigma_{a}z}\overline{J_\a^{\sigma_{a}z}}(k/2+s)y^s\\
\end{align}
Let $h(z)=(z-\overline z)\frac{\Omega_{V,1}(\sigma_\a z)}{j(\sigma_\a,z)^2 }$. Using the series-representation \eqref{D-total} and \eqref{eq:RT-commutator} we see that 
\begin{align}
    R_kD^{**}_\a(z,s,k)=\sum_{\g\in \G_\a\backslash \Gamma}&T_{k+2, \sigma_{\a}^{-1}\g} (h(z)J_\a^{\sigma_\a z}(\Omega_V)\overline{J_\a^{\sigma_\a z}(\Omega_W)}y^s)\\ &+(k/2+s)D^{**}_\a(z,s,k)
\end{align}
Using that $h(\sigma_\a^{-1}\g,z)=j_{\sigma_a^{-1}\g}(z)^2(z-\overline{z})\Omega_{V,1}(z)$ proves \ref{raising}.

The claim in \ref{lowering} is proved analogously. 

To see \ref{laplace} we write, for $s\in\C$,  $\lambda(s)=s(1-s)$ and see that 
\begin{align}
(\Delta_k+\lambda(s))&D^{**}_\a(z,s,k)\\ &=-R_{k-2}L_{k}D^{**}_\a(z,s,k)+(\lambda(s)-\lambda(k/2))D^{**}_\a(z,s,k)
\end{align}
We note, by \eqref{ref:kill}, that $R_{-2}((z-\overline z)\overline{\Omega_{W,1}})=0$. Using \ref{raising} and \ref{lowering} together with this we find that 
\begin{align}
R&_{k-2}L_{k}D^{**}_\a(z,s,k)\\
=&R_{k-2}((z-\overline z)\overline{\Omega_{W,1}}D^{**}_\a(z,s,k)+(k/2-s)D^{**}_\a(z,s,k-2))\\
=&R_{-2}((z-\overline z)\overline{\Omega_{W,1}})D^{**}_\a(z,s,k)+(z-\overline z)\overline{\Omega_{W,1}}R_kD^{**}_\a(z,s,k)\\
&+(k/2-s)((z-\overline z)\Omega_{V,1}D^{**}_\a(z,s,k-2)+((k-2)/2+s)D^{**}_\a(z,s,k))\\
=&(z-\overline z)\overline{\Omega_{W,1}}((z-\overline z)\Omega_{V,1}D^{**}_\a(z,s,k)+(k/2+s)D^{**}_\a(z,s,k+2))\\
&+(k/2-s)((z-\overline z)\Omega_{V,1}D^{**}_\a(z,s,k-2)+((k-2)/2+s)D^{**}_\a(z,s,k)).
\end{align}
A straightforward computation verifies that \begin{equation}(k/2-s)((k-2)/2+s)-(\lambda(s)-\lambda(k/2))=0,\end{equation} and using this and that $\overline{\Omega_{W,1}}$ and $\Omega_{V,1}$ commute,  we arrive at \ref{laplace}.

We prove \ref{laplace-better} by isolating the last term in \ref{raising} and \ref{lowering} and inserting them in \ref{laplace}.
\end{proof}

Proposition \ref{prop:LaplaceD*} succinctly expresses how the Laplacian and the raising and lowering operators act on all the  coefficients of the total twisted Eisenstein series  $D^{**}_\a{(z,s,k,\Omega_V, \Omega_W)}$ at the same time. For later use it is also useful to spell out how they act on the coefficients directly. 

We define an operator $\delta:V^*\backslash\{\e\}\to V^*$ (and similarly on $W$) which deletes the first letter, i.e. if $v=v_1v_2\ldots v_n$, $v_i\in V$, is a word of length $l(v):=n>0$, then $\delta v$ is the word of length $l(\delta v)=n-1$ given by 
\begin{equation} 
    \delta v=v_2\ldots v_n\in V^*.
\end{equation}
\begin{cor}\label{cor: laplaceDcoeff}For $v\in V^*$, $w\in W^*$  we have 
\begin{enumerate}[label=(\roman*)]
    \item \label{item:laplace-one}
$\begin{aligned}[t]
    (\Delta_k+s(1-s))&D^{vw}_\a(z,s,k)=-(z-\overline z)^2f_{v_1}\overline{f_{w_1}}D^{\delta v \delta w}_\a(z,s,k)\\ 
        & -(k/2+s)(z-\overline z)\overline{f_{w_1}}D^{v\delta w}_\a(z,s,k+2)\\ &-(k/2-s)(z-\overline z)f_{v_1}D^{\delta v w}_\a(z,s,k-2),
\end{aligned}$
\item\label{item:laplace-two}
$\begin{aligned}[t]
    (\Delta_k+s(1-s))&D^{vw}_\a(z,s,k)
        =(z-\overline z)^2f_{v_1}\overline{f_{w_1}}D^{\delta v \delta w}_\a(z,s,k)\\ 
        & -(z-\overline z)\overline{f_{w_1}}R_kD^{v\delta w}_\a(z,s,k)\\ &-(z-\overline z)f_{v_1}L_kD^{\delta v w}_\a(z,s,k).
\end{aligned}$
\end{enumerate}
\end{cor}
Here it is understood that if we write $\delta \e$ in any of the indices the corresponding term is absent.

\subsection{Meromorphic continuation of
the automorphic twisted Eisenstein series} 
\label{sec:continuation-etc}

In this section we first briefly sketch a proof of Theorem \ref{thm:analytic} below which establishes some basic but important properties of $D^{vw}_\a(z,s,k)$. The claim about meromorphic continuation 
is proved also in \cite{ChintaHorozovOSullivan:2019}, but our proof is new. We then go on to find the polar structure at $s=1$ of the meromorphic continuation. 

\begin{thm}\label{thm:analytic} For $(v,w)\in V^*\times W^*$ the function $D^{vw}_\a(z,s,k)$ admits meromorphic continuation to $\Re(s)>1/2$ with all poles $s_0$ satisfying that $s_0(1-s_0)$ is in the spectrum of  $-\Delta_k$. Furthermore $D^{vw}_\a(z,s,k)$ grows at most polynomially in $s$ in any vertical strip uniformly for $z$ in any compact set. For $(v,w)\neq (\e,\e)$ the function $D^{vw}_\a(z,s,k)$ is in $L^2(\G,k)$ away from its poles. 
\end{thm}
\begin{proof}[Sketch of proof]The proof is an induction in $l(v)+l(w)=q$. For $q=0$ we have $D^{vw}_\a(z,s,k)=E_\a(z,s,k)$ and all the statements are well-known. For the polynomial growth we refer to \cite[Lemma 3.1]{PetridisRisager:2004a} for $k=0$. The $k\neq 0$ case is proved analogously.

For a general $q$ we apply the resolvent to Corollary \ref{cor: laplaceDcoeff} \ref{item:laplace-one} and get 
\begin{align}
\label{eq:Resolvent-D}
\begin{split}
    D^{vw}_\a(z,s,k)=&R(s,k)\left[(z-\overline z)^2f_{v_1}\overline{f_{w_1}}D^{\delta v \delta w}_\a(z,s,k)\right.\\ 
        & +(k/2+s)(z-\overline z)\overline{f_{w_1}}D^{v\delta w}_\a(z,s,k+2)\\ &+\left.(k/2-s)(z-\overline z)f_{v_1}D^{\delta v w}_\a(z,s,k-2)\right],
\end{split}
\end{align}
Note that since $f_{v_1},f_{w_1}$ are cuspidal the right-hand side of Corollary \ref{cor: laplaceDcoeff} \ref{item:laplace-one} is in $L^2(\G,k)$ and we saw in \eqref{square-integrable} that $D^{vw}_\a(z,s,k)$ is in $L^2(\G,k)$ when $\Re(s)>1$. Now the right-hand side of \eqref{eq:Resolvent-D} is meromorphic in $\Re(s)>1/2$ by induction and the properties of the resolvent, and the possible poles are as claimed. The claim about growth in vertical lines is proved as in  \cite[Lemma 3.2]{PetridisRisager:2004a} using $L^2$-estimates and the Sobolev embedding theorem. We omit the details. 
\end{proof}
In order to understand the average growth of the noncommutative modular symbols we need better understanding of the pole $D^{vw}_\a(z,s,k)$. Recall that $D^{\e\e}_\a(z,s,k)=E_\a(z,s,k)$ has a pole at $s=1$ only if $k=0$, and in this case the pole is simple with residue $\vol\GmodH^{-1}$. In order to understand the general automorphic twisted Eisenstein series we need the following result:

\begin{prop}\label{prop:term-vanishes}Let $f\in S_2(\G)$ and $D\in L^2(\G,0)$. Then 
\begin{equation}
    \inprod{(z-\overline z)\overline f R_0D }{1}=\inprod{(z-\overline z)f L_0 D }{1}=0.
\end{equation}
Furthermore
\begin{equation}
    \inprod{(z-\overline z)\overline f R_0E(z,s,0) }{1}=\inprod{(z-\overline z)f L_0 E(z,s,0) }{1}=0.
\end{equation}
\end{prop}
\begin{proof}
We have 
\begin{equation}
    \inprod{(z-\overline z)\overline f R_0D }{1}=\inprod{D }{L_2(\overline z- z)f }=0
\end{equation}
where we have used \eqref{ref:adjoint} and \eqref{ref:kill}. This proves that the first expression in the proposition equals zero. That the second is also zero is proved analogously. 

Proving that the same thing holds for $D=E(z,s,0)$ follows from observing that since the decay of the cusp form $f$ makes the underlying integral convergent the integration by parts leading to \eqref{ref:adjoint} is still valid in this case. 
\end{proof}

\begin{prop}\label{prop:unmixed-regular} Let $\e\neq v\in V^*$. Then $D^{v\e}(z,s,0)$ and $D^{v\e}(z,s,0)$ are regular at $s=1$.
\end{prop}
\begin{proof}
Note that $\overline{D^{v\e}(z,\overline s,-k)}=D^{\e v}(z, s,k)$ so it suffices to prove that $D^{v\e}(z,s,0)$ is regular at $s=1$ which we do by induction in $l(v)=l$. By using Corollary \ref{cor: laplaceDcoeff} \ref{item:laplace-two} we have 
\begin{equation}\label{useful-expression}
D^{v\e}(z,s,0)=R(s,0) ((z-\overline z)f_{v_1}L_0D^{\delta v \e }_\a(z,s,0)).
\end{equation}
For $l=0$ we note that since $L_0$ differentiates it annihilates the singular part of $D^{\delta v \e }_\a(z,s,0)=E_\a(z,s,0)$, so $L_0D^{\delta v \e }_\a(z,s,0)$ is regular (alternatively use \eqref{ref:Eisenstein-properties}). By \eqref{eq:resolvent-splitting} it follows that the potential pole of $D^{v\e}(z,s,0)$ at $s=1$ can be at most simple, but Proposition \ref{prop:term-vanishes} shows that the singular part vanishes identically which proves the claim for $l=1$. 

The general case again uses \eqref{useful-expression}. By induction  $(z-\overline z)f_{v_1}L_0D^{\delta v \e }_\a(z,s,0)$ is regular at $s=1$. By \eqref{eq:resolvent-splitting} the singular part of $D^{v\e}(z,s,0)$ equals that of the expression $\inprod{(z-\overline z)f_{v_1}L_0D^{\delta v \e }_\a(z,s,0)}{1}\vol\GmodH^{-1}(s-1)^{-1}$ which vanishes by Proposition \ref{prop:term-vanishes}.
\end{proof}

\begin{thm}\label{thm:pole-orders}Let $(v,w)\in V^*\times W^*$ with $l(v)+l(w)=q$.
\begin{enumerate}
    \item  If $q$ is even then $D^{vw}_\a(z,s,0)$ has a pole at $s=1$ of order at most $q/2+1$, and this pole order is attained only if $l(v)=l(w)$. In this case corresponding coefficient equals 
\begin{equation}
 \frac{4^{l(v)}}{\vol\GmodH^{l(v)+1}}  \prod_{i=1}^{l(v)}\inprod{yf_{v_i}}{yf_{w_i}}.
\end{equation}

\item If $q$ is odd then $D^{vw}_\a(z,s,0)$ has a pole at $s=1$ of order at most $(q-1)/2$ 
\end{enumerate}
\end{thm}
\begin{proof} For $q=0$ this follows from the known properties of  $E_\a(z,s,0)$ (See section \ref{sec:maass-background}). For $q=1$ this follows from Proposition \ref{prop:unmixed-regular}. 

To run an inductive argument we assume that the claim has been proved up to $q$ odd. Assume that $l(v)+l(w)=q+1$ which is now even. By Corollary \ref{cor: laplaceDcoeff} \ref{item:laplace-two} we have
\begin{align}
    D^{vw}_\a(z,s,0)
        =-&R(s,0)\left[(z-\overline z)^2f_{v_1}\overline{f_{w_1}}D^{\delta v \delta w}_\a(z,s,0)\right.\\ 
\label{eq:recursion-k0}        & -(z-\overline z)\overline{f_{w_1}}R_0D^{v\delta w}_\a(z,s,0) \left.-(z-\overline z)f_{v_1}L_0D^{\delta v w}_\a(z,s,0)\right].
\end{align}
By the induction hypothesis and \eqref{eq:resolvent-splitting} the first term is at most of order $(q-1)/2+1+1$ and this order is attained only if $l(\delta v)=l(\delta w)$ and the corresponding coefficient equals
\begin{equation}
    -\frac{4^{l(\delta v)}}{\vol\GmodH^{l(\delta v)+1}}  \left(\prod_{i=2}^{l(v)}\inprod{yf_{v_i}}{yf_{w_i}}\right)\inprod{(z-\overline z)f_{v_1}}{(z-\overline z)f_{w_1}}
\end{equation}
Using that the last inner product equals $-4\inprod{yf_{v_1}}{yf_{w_1}}$ the coefficient corresponding to $(s-1)^{(q+1)/2+1}$. 

To analyze the second and third term  the induction hypothesis and \eqref{eq:resolvent-splitting} implies that this is of order at most $(q-1)/2+1$. This proves the claim for $l(v)+l(w)=q+1$. 

Assume now that $l(v)+l(w)=q+2$ which is odd, and want to argue that $D^{vw}_\a(z,s,0)$ has a pole at $s=1$ of order at most $((q+2)-1)/2$ We use again \eqref{eq:recursion-k0}, \eqref{eq:resolvent-splitting}, and the induction hypothesis. We note that the first term on the left-hand side has pole order at most $(q-1)/2+1$ agrees with the claim. 
To analyze the second term $R(s,0)(z-\overline z)\overline{f_{w_1}}R_0D^{v\delta w}_\a(z,s,0)$ we note that $D^{v\delta w}_\a(z,s,0)$ has a pole of order at most order $(q+1)/2+1$ but if this is attained the leading coefficient is constant so since $R_0$ differentiates $R_0D^{v\delta w}_\a(z,s,0)$ has a pole of order at most $(q+1)/2$. By \eqref{eq:resolvent-splitting} the resolvent $R(s,0)$ raises this pole order by at most one, but only if 
\begin{equation}
    \inprod{(z-\overline z)\overline{f_{w_1}}R_0D^{v\delta w}_\a(z,s,0)}{1}
\end{equation}
is not identically zero. But Proposition \ref{prop:term-vanishes} implies that this expression is indeed identically zero which proves the claim, and finishes the induction.

\end{proof}

\begin{thm} \label{thm:pole-order-bounds}The function $D^{vw}_\a(z,s,k)$ has a pole at $s=1$ of order at most $\min(l(v),l(w))+1$. If $k\neq 0$ and $l(v)=l(w)$ the pole order is strictly less than $\min(l(v),l(w))+1$. 
\end{thm}
\begin{proof}We first prove the result for $k=0$. We start again with an induction in $l(v)+l(w)=q$. For $q=0$ this is a well-known property of the weight 0 Eisenstein series.  
Assume now that the claim has been proved for all $l(v)+l(w)\leq q$. On the right-hand side of the expression \eqref{eq:recursion-k0} the first term has pole order at most $\min(l(v)-1, l(w)-1)+1+1=\min(l(v),l(w))+1$. The second term has pole order at most $\min(l(v),l(w)-1)+1+1$ but this order is attained only if $\inprod{(z-\overline z)\overline{f_{w_1}}R_0D^{v\delta w}_\a(z,s,0)}{1}$ is not identically zero. But by Proposition \ref{prop:term-vanishes} it is identically zero. So this term also has pole order less than or equal to $\min(l(v),l(w))+1+1$. The third term is analyzed in the same way. This finishes the proof when $k=0$. 

For $k>0$, and $l(v)=l(w)$ we use the raising operator: From Proposition \ref{prop:LaplaceD*} \ref{raising} we find  that
\begin{equation}\label{raising-induction}
    D_\a^{v,w}(z,s,k'+2)=\frac{1}{k'/2+s}\left(R_{k'}D_\a^{vw}(z,s,k')-(z-\overline z)f_{v_1}D_\a^{\delta v,w}(z,s,k')\right)
\end{equation}
For $k'=0$ we find by the statement above that last term has pole order at most $l(v)$. The expression $D_\a^{vw}(z,s,0)$ a priori has a pole order of $l(v)+1$, but by Theorem \ref{thm:pole-orders} the corresponding coefficient in the Laurent expansion is constant in $z$ so $R_{0}D_\a^{vw}(z,s,0)$ has a pole order of at most $l(v)$ which proves the claim for $k=2$. The claim for $k>2$ now follows by using \ref{raising-induction} recursively.   
 The claim for $k<0$ is proved analogously by using the lowering operator. 
\end{proof}

We can now use the results about the automorphic twisted Eisenstein series $D^{**}_\a(z,s,k)$ to  find the polar structure the higher order twisted Eisenstein series $E_\a^{**}(z,s,k)$. Using Remark \ref{rem:back-forth},  \eqref{ref:back-forth} and  Theorems \ref{thm:analytic}, \ref{thm:pole-orders}, and \ref{thm:pole-order-bounds} we immediately arrive at the following theorem:
\begin{thm}\label{E-properties}
Let $(v,w)\in V^*\times W^*$. Then 
\begin{enumerate}[label={(\roman*})]
    \item $E^{vw}_\a(z,s,k)$ admits meromorphic continuation to $\Re(s)>1/2$ with all poles $s_0$ satisfying that $s_0(1-s_0)$ is in the spectrum of  $-\Delta_k$.
    \item $E^{vw}_\a(z,s,k)$ grows at most polynomially in $s$ in any vertical strip uniformly for $z$ in any compact set.
    \item The pole at $s=1$ has order at most $\min(l(v),l(w))+1$. 
\end{enumerate}
Assume that $l(v)=l(w)$ such that $\min(l(v),l(w))+1=l(v)+1$. Then
\begin{enumerate}[resume, label={(\roman*})]
    \item if $k=0$, the pole order at $s=1$ equals $l(v)+1$ and the corresponding coefficient in the Laurent expansion equals
    \begin{equation}
    \frac{4^{l(v)}}{\vol\GmodH^{l(v)+1}}  \prod_{i=1}^{l(v)}\inprod{yf_{v_i}}{yf_{w_i}}.
    \end{equation}
    \item if $k\neq 0$ the pole order at $s=1$ is strictly less than $l(v)+1$.
\end{enumerate}
\end{thm}
Note that even if we only defined $E^{**}_\a{(z,s,k,\Omega_V, \Omega_W)}$ for $k$ even the same expression defines a function for $k$ odd. But since $\g$ and $-\g$ gives opposite contributions this function is identically zero. Hence Theorem \ref{E-properties} holds trivially for $k$ odd.

\subsection{Fourier coefficients of 
higher order twisted Eisenstein series} 
In this subsection we compute the Fourier coefficients  of the higher order twisted Eisenstein series $E^{**}_\a{( z,s,k,\Omega_V, \Omega_W)}$. The key to this is the following proposition:
\begin{prop}\label{cuspidal-invariance-left-right}
For $v\in V^*$ and $\g, \g_1, \g_2\in \G$ with $\g_1,\g_2$ parabolic we have 
\begin{equation}
J_{a}^{\g_1\g\g_2 b}(v)=J_{a}^{\g b}(v). 
\end{equation}
\end{prop}
\begin{proof}
By \cite[Cor. 2.3]{ChintaHorozovOSullivan:2019} we have, for $\g_i$ parabolic,  that  $$J_c^{\g_i d}(v)=J_{\g_ic}^{  d}(v)=J_c^{d}(v)$$ for any $c,d\in \H^*$. Combining this with the $\G$-invariance from Proposition \ref{prop:symbols-properties} \eqref{gamma-invariance} we have $J_{\g c}^{  \g d}(v)=J_{c}^{  d}(v)$  the result follows easily.
\end{proof}
Once Proposition \ref{cuspidal-invariance-left-right} is established the computation of the Fourier coefficients proceeds analogously to the computation for the standard Eisenstein series: 
We find using Proposition \ref{cuspidal-invariance-left-right} and \eqref{j-properties} the expression
\begin{align}
   j&_{\sigma_{\b}}^{-k}(z)E^{**}_\a{(\sigma_\b z,s,k)}-\delta_{\a\b}y^s\\
   &\quad =\sum_{ \sigma_\a^{-1}\sigma_\b\neq \g\in \Gamma_{\!\infty}\!\backslash \sigma_{\a}^{-1}\G\sigma_{\b}/\Gamma_{\!\infty}} \!\!\!\!\!\!\!\!\!\!\!\!\!\!\!\!\!\!\!\!\!J_a^{\sigma_\a\g\sigma_\b^{-1} a}(\Omega_V)\overline{J_a^{\sigma_\a\g\sigma_\b^{-1} a}(\Omega_W)}\sum_{n \in \Z}j_{\g}^{-k}(z+n)\Im(\g(z+n))^s
\end{align}
where we have used that the term coming from the identity term in $\G$ satisfies $J_a^{a}=1$.
The inner sum $S_k(x)$ can be analyzed as for the standard Eisenstein series either by using Poisson summation (as in \cite[Sec. 3.4]{Iwaniec:2002a}) or by computing its Fourier coefficients directly as
\begin{align}
    \int_{0}^1S_k(x)&e(-nx)dx=\int_{-\infty}^\infty j_{\g}^{-k}(z)\Im(\g(z))^s e(-nx)dx\\
    &=\int_{-\infty}^\infty \left(\frac{c(x+iy)+d}{\abs{c(x+iy)+d}}\right)^{-k}\frac{y^s}{\abs{c(x+iy)+d}^{2s}} e(-nx)dx\\
    &=\frac{y^s}{\abs{c}^{2s}}\int_{-\infty}^\infty\left(\frac{(x+iy)+\frac{d}{c}}{\abs{(x+iy)+\frac{d}{c}}}\right)^{-k}\frac{1}{\abs{(x+iy)+\frac{d}{c}}^{2s}} e(-nx)dx\\
    &=\frac{y^s}{\abs{c}^{2s}}e(nd/c)\int_{-\infty}^\infty\left(\frac{x+iy}{\abs{x+iy}}\right)^{-k} \frac{e(-nx)}{(x^2+y^2)^s} dx\\
    &=\frac{e(nd/c)}{\abs{c}^{2s}i^k}\begin{cases} 
    (4y)^{1-s}\frac{\Gamma(2s-1)}{\Gamma(s+\frac{k}{2})\Gamma(s-\frac{k}{2})} & n=0\\
    \frac{\abs{n}^{s-1}\pi^s}{\Gamma(s+\sgn(n)\frac{k}{2})}W_{\frac{k}{2}\sgn{n},s-\frac{1}{2}}(4\pi \abs{n}y) & n\neq 0
    \end{cases}
\end{align}
where we have evaluated the integral as in \cite[p.7-8]{Young:2019}. If $k=0$ the Whittaker function simplifies to a $K$-Bessel function such that in this case
\begin{align}
    \int_{0}^1S_0(x)e(-nx)dx
    =\frac{e(nd/c)}{\abs{c}^{2s}}
    \begin{cases} 
    y^{1-s}\sqrt{\pi}\frac{\Gamma(s-\frac{1}{2})}{\Gamma(s)} & n=0\\
    \frac{2\abs{n}^{s-1/2}\pi^s}{\Gamma(s)}\sqrt{y}K_{s-\frac{1}{2}}(2\pi \abs{n}y) & n\neq 0
    \end{cases}
\end{align}

We can now use the usual double coset decomposition \cite[Thm 2.7]{Iwaniec:2002a} to write the full Fourier decomposition. We do this only when $k=0$ as the the $k$ dependence is only in the Whittaker function so that the arithmetic part is independent of $k$. We have

\begin{equation}
        E^{**}_\a{(\sigma_\b z,s,0)}=\delta_{\a\b}y^s +\varphi_{\a\b}^{**}(s)y^{1-s} +\sum_{n\neq 0}\varphi_{\a\b}^{**}(s,n)\sqrt{y}K_{s-1/2}(2\pi \abs{n}y)e(nx),
\end{equation}
where 

\begin{equation}
\varphi_{\a\b}^{**}(s):=\varphi_{\a\b}^{**}(s,0):=\sqrt{\pi}\frac{\Gamma(s-1/2)}{\Gamma(s)}L_{\a, \b}^{**}(s,0,0),
\end{equation}
\begin{equation}
\varphi_{\a\b}^{**}(s,n)=\frac{2\abs{n}^{s-1/2}\pi^s}{\Gamma(s)}L_{\a, \b}^{**}(s,0,n),
\end{equation}
and 
\begin{equation}
    L_{\a, \b}^{**}(s,m,n)=\sum_{c>0}\frac{S_{\a\b}^{**}(m,n,c)}{c^{2s}}.
\end{equation}

Here the sum over $c$ is over all  positive lower left entries of  $\sigma_{\a}^{-1}\G\sigma_\b$ and
\begin{equation}
    S_{\a\b}^{**}(m,n,c)=
\sum_{\g=\begin{psmallmatrix}a & *\\c&d 
    \end{psmallmatrix}\in \sigma_{\a}^{-1}\G\sigma_\b } 
    J_a^{\sigma_\a\g\sigma_\b^{-1} a}(\Omega_V)\overline{J_a^{\sigma_\a\g\sigma_\b^{-1} a}(\Omega_W)}e(\frac{am+dn}{c})
\end{equation}

We note that by comparing coefficients we find that for $(v,w)\in V^*\times W^*$ we have
\begin{align}
    E^{vw}_\a{(\sigma_\b z,s,0)}=\delta_{v\e}\delta_{w\e}&\delta_{\a\b}y^s +\varphi_{\a\b}^{vw}(s)y^{1-s}\\ & +\sum_{n\neq 0}\varphi_{\a\b}^{vw}(s,n)\sqrt{y}K_{s-1/2}(2\pi \abs{n}y)e(nx)
\end{align}
where
\begin{equation}
\varphi_{\a\b}^{vw}(s):=\varphi_{\a\b}^{vw}(s,0):=\sqrt{\pi}\frac{\Gamma(s-1/2)}{\Gamma(s)}L_{\a, \b}^{vw}(s,0,0),
\end{equation}
\begin{equation}
\varphi_{\a\b}^{vw}(s,n):=\frac{2\abs{n}^{s-1/2}\pi^s}{\Gamma(s)}L_{\a, \b}^{vw}(s,0,n).
\end{equation}
and 
\begin{align}
    L_{\a, \b}^{vw}(s,m,n) &=\sum_{c>0}\frac{S_{\a\b}^{vw}(m,n,c)}{c^{2s}}\\
        S_{\a\b}^{vw}(m,n,c) &=
\sum_{\g=\begin{psmallmatrix}a & *\\c&d 
    \end{psmallmatrix}\in \sigma_{\a}^{-1}\G\sigma_\b } 
    I_a^{\sigma_\a\g\sigma_\b^{-1} a}(v)\overline{I_a^{\sigma_\a\g\sigma_\b^{-1} a}(w)}e(\frac{am+dn}{c})
\end{align}

Note that the sum over $c$ is uniformly and absolutely convergent for compact subsets of $\Re(s)>1$ by standard bounds in the theory of Eisenstein series and \cite[Cor. 2.6]{ChintaHorozovOSullivan:2019}.

\begin{thm}\label{thm:dirichlet-series-prop}Let $(v,w)\in V^*\times W^*$.
The function
$L^{vw}_{\a\b}(s,0,0)$
originally defined for $\Re(s)>1$ 
\begin{enumerate}
    \item admits meromorphic continuation to $Re(s)>1/2$ with all poles $s_0$ satisfying that $s_0(1-s_0)$ is in the spectrum of  $-\Delta_k$. 
    Moreover $L^{vw}_{\a\b}(s,0,0)$ 
    \item grows at most polynomially in $s$ in any vertical strip.

    \item The pole at $s=1$ is of order at most $\min(l(v),l(w))+1$. 
    \item If $l(v)=l(w)$ such that $\min(l(v),l(w))+1=l(v)+1$ the pole order at $s=1$ equals $l(v)+1$ and the corresponding coefficient in the Laurent expansion equals
    \begin{equation}
    \frac{1}{\pi}\frac{4^{l(v)}}{\vol\GmodH^{l(v)+1}}  \prod_{i=1}^{l(v)}\inprod{yf_{v_i}}{yf_{w_i}}.
    \end{equation}
\end{enumerate}
\end{thm}
\begin{proof}
This follows directly from Theorem \ref{E-properties} by considering 
\begin{equation}
    \int_0^1E^{vw}(z,s)dx=\delta_{v\e}\delta_{w\e}\delta_{\a\b}y^s +\sqrt{\pi}\frac{\Gamma(s-1/2)}{\Gamma(s)}L_{\a, \b}^{vw}(s,0,0)y^{1-s}
\end{equation}
for any fixed $y$, combined with the Stirling asymptotics on the Gamma function, and that $\Gamma(1/2)=\sqrt{\pi}.$ 
\end{proof}

\begin{rem}\label{rem:translation to T_ab}
Consider, as in \cite[Sec 2.2]{PetridisRisager:2018a},  the infinite set
\begin{equation}
    T_{\a\b}=\left\{\frac{a}{c} \bmod 1, \begin{pmatrix}a&b\\c&d\end{pmatrix}\in \G_\infty\backslash\sigma_{\a}^{-1}\G\sigma_\b\slash \G_\infty\textrm{ and }c>0 \right\}\subseteq \R/\Z
\end{equation}
Given an $r\in T_{\a\b}$ we have \cite[Cor 2.3]{PetridisRisager:2018a} that there exist a unique number $c(r)>0$ and unique cosets $a(r)\bmod c(r)$ $a(r)\bmod c(r)$ such $a(r)d(r)=1\bmod c(r)$ and $$r=\frac{a(r)}{c(r)} \bmod 1.$$

For $v\in  V^*$ and $r\in T_{\a\b}$ we see that

\begin{equation}\label{nice-map}
    I_{\b}^{\sigma_\a r }(v)= I_{\b}^{\sigma_\a \gamma \sigma_b^{-1}\b }(v)
\end{equation}
where $\g\in \sigma_{\a}^{-1}\G\sigma_\b$ with positive lower left entry and satifying $\g(\infty)=r$.

Now we note that if we set $a=\b$ then
\begin{equation}
L_{\a\b}^{vw}(s,m,n)=\sum_{r\in T_{\a\b}}\frac{I_{\b}^{\sigma_\a r }(v)\overline{I_{\b}^{\sigma_\a r }(w)}e(mr+n\overline r)}{c(r)^{2s}}
\end{equation}
where if $r=a(r)/c(r) \bmod 1$ then we have defined $\overline r=d(r)/c(r) \bmod 1$. Note that we have put $a=\b$ in the definition of $S_{\a\b}^{vw}(m,n,c)$. 

\end{rem}

\begin{rem}\label{rem:added-exponential}
One can prove a result analogous to Theorem \ref{thm:dirichlet-series-prop} for the functions $L^{vw}_{\a\b}(s,0,n)$, and $L^{vw}_{\a\b}(s,n,0)$ In this case the functions grows at most polynomially in $s$ and $n$, and the pole order when $l(v)=l(w)$ is strictly less than $l(v)+1$. The reason we cannot simply copy the above proof to prove this claim for $L^{vw}_{\a\b}(s,0,n)$ is that we do not have good enough control on $\sqrt{y}K_{s-1/2}(2\pi \abs{n}y)/\G(s)$. We need good lower bounds. One way around this is to realize $L^{vw}_{\a\b}(s,0,n)$ as an expression involving inner products of $D^{vw}_\a(z,s,0)$ with Poincare series twisted with noncommutative modular symbols. In the case of classical modular symbols this is carried out in \cite[Sec 6.]{PetridisRisager:2018a}, and we leave the details to the reader.  It is straightforward to verify
\begin{equation}
    S_{\a\b}^{vw}(m,n,c)=\overline{S_{\a\b}^{wv}(-m,-n,c)},
\end{equation}and
\begin{equation}
    S_{\a\b}^{wv}(m,n,c)=(-1)^{l(v)+l(w)}S_{\b\a}^{v^rw^r}(-n,-m,c),
\end{equation}
where in the second equation we have used $\g\mapsto \g^{-1}$ and  the reversal of path formula \cite[Thm 3.19]{GilFresan:2017}.
Using this we find that $$L^{vw}_{\a\b}(s,n,m)=(-1)^{l(v)+l(w)}L^{v^rw^r}(s,-n,-m),$$ so any result about $L^{vw}_{\a\b}(s,0,n)$, may be translated into a result about $L^{vw}_{\a\b}(s,n,0).$
\end{rem}

\section{Moments of noncommutative modular symbols}
\label{sec:Moments-computation}
In the previous section we found the polar structure of the two Dirichlet series $E_a^{vw}(z,s,0)$and $L_{\a\b}^{vw}(s,0,0)$. We now translate that into average results about noncommutative modular symbols. 

Consider the two finite sets
\begin{align}
    (\G_\a\backslash\G)(M)&:=\{ \g\in \G_\a\backslash\G\vert\, \Im(\sigma_\a^{-1}\g z)^{-1}\leq M\}\\
    T_{\a\b}(M)&:=\{ r\in T_{\a\b}\vert\, c(r)^2\leq M\}.
\end{align}
These will eventually correspond to two different orderings of the noncommutative modular symbols.  Theorem \ref{E-properties} and  Theorem \ref{thm:dirichlet-series-prop} in combination with Remark \ref{rem:translation to T_ab} allows us to use a complex integration argument (See e.g. \cite[Appendix A]{Nordentoft:2018}) to conclude the following:

Let $v\in V^*$, $w\in W^*$ and assume that $l(v)=l(w)$. We define
\begin{equation}\label{def:B}
    B(v,w)=\frac{\displaystyle4^{l(v)}\prod_{i=1}^{l(v)}\inprod{yf_{v_i}}{yf_{w_i}}}{\displaystyle l(v)!\vol\GmodH^{l(v)+1}}, 
\end{equation} if $l(v)=l(w)$, and $B(v,w)=0$ if $l(v)\neq l(w)$
Then we have the following asymptotical result: 

\begin{thm}\label{thm:asymptotics}
Let $(v,w)\in V^*\times W^*$,
If $l(v)=l(w)$ then 
\begin{align}
\sum_{\g\in (\G_\a\backslash\G)(M)}I_a^{\g a}(v)\overline{I_a^{\g a}(w)}
&=B(v,w) M\log^{l(v)}M+O(M\log^{l(v)-1}M)\\
\sum_{r\in T_{\a\b}(M)}I_\b^{\sigma_\a r}(v)\overline{I_\b^{\sigma_\a r}(w)}
&=\frac{1}{\pi}B(v,w) M\log^{l(v)}M+O(M\log^{l(v)-1}M)
\end{align}
For $v,w$ of any lengths the sums on the left are $O(M\log^{\min{(l(v),l(w))}}M)$.
\end{thm}

\begin{rem}
It is possible to refine the above asymptotic to considering only $\gamma$ satisfying  $\frac{1}{2\pi}\hbox{Arg}(j(\sigma_\a^{-1}\g,z)\in I$ resp. $r\in I$ for some interval of $\R/\Z$. In this case the right hand side simply gets multiplied by $\abs{I}$ This can be interpreted as a sort of homogeneity in angle sectors resp. subintervals of $\R\backslash\Z$. For the second sum this would follow from Remark \ref{rem:added-exponential} which allows to get upper bounds on sums of the form  
\begin{equation}
    \sum_{r\in T_{\a\b}(M)} I_\b^{\sigma_\a r}(v)\overline{I_\b^{\sigma_\a r}(w)} e(nr)
\end{equation} 
when $n\neq 0$.
Using an approximation of the indicator function of $I$ by exponentials one finds asymptotics for 
\begin{equation}
   \sum_{r\in T_{\a\b}(M)\cap I}I_\b^{\sigma_\a r}(v)\overline{I_\b^{\sigma_\a r}(w)}
    \end{equation}
For the first sum in Theorem \ref{thm:asymptotics} this follows in the same fashion by considering all weights $k$ in Theorem \ref{E-properties}
\end{rem}

We can now use Theorem \ref{thm:asymptotics} and the shuffle product formula \eqref{shuffle-product-formula} to compute the result if we replace the noncommutative modular symbols in Theorem \ref{thm:asymptotics} by powers of the same.
We  define
\begin{equation}
C_{n,m}(v,w)=\sum_{u_1,u_2\in V^*}c_{v^{\shuffle n}}(u_1)c_{w^{\shuffle m}}(u_2)B(u_1,u_2)
\end{equation}
if $nl(v)=ml(w)$, and zero otherwise. 

\begin{thm}\label{thm:moments}
Let $(v,w)\in V^*\times W^*$,
If $nl(v)=ml(w)$ then 
\begin{align}
\sum_{\g\in (\G_\a\backslash\G)(M)}I_a^{\g a}(v)^n\overline{I_a^{\g a}(w)}^m
&=C_{n,m}(v,w) M\log^{nl(v)}M+O(M\log^{nl(v)-1}M)\\
\sum_{r\in T_{\a\b}(M)}I_\b^{\sigma_\a r}(v)^n\overline{I_\b^{\sigma_\a r}(w)}^m
&=\frac{1}{\pi}C_{n,m}(v,w) M\log^{ml(v)}M+O(M\log^{nl(v)-1}M)
\end{align}
For $v,w$ of any lengths the sums on the left are $O(M\log^{\min{(nl(v),ml(w))}}M)$.
\end{thm}
\begin{proof}
This follows directly from Theorem \ref{thm:asymptotics} and \eqref{shuffle-product-formula}.
\end{proof}

Fix a $v\in V$.We are interested in the the limiting distribution of 
\begin{enumerate}
    \item $I_{a}^{\g a}(v)$ for fixed $a\in \H^*, z\in \H$ for $\g$ in $(\G_\a\backslash\G)(M)$ as $M\to\infty$. 
    \item  $I_{\b}^{\sigma_{\a}r}(v)$ for fixed $\a,\b$ with $r\in T_{\a\b}(M)$, as $M\to\infty.$   
\end{enumerate}
We renormalize these to have finite moments as follows: For $C\subseteq\C$ a sufficiently nice subset of the complex plane consider the following two complex random variables:  
\begin{equation}
X_M^v(C)=\frac{\#\{\g\in (\G_\a\backslash\G)( M)\vert\, \left(\frac{\vol{\GmodH}}{4\log(\Im(\sigma_\a^{-1}\g z)^{-1})}\right)^{l(v)/2} I_{a}^{\g a}(v)\in C\}}{\#(\G_\a\backslash\G)(M)}
\end{equation}
\begin{equation}
Y_M^v(C)=\frac{\#\{r\in T_{\a\b}(M)\vert\, \left(\frac{\vol{\GmodH}}{4\log(c(r)^2)}\right)^{l(v)/2} I_{\b}^{\sigma_{\a}r}(v)\in C\}}{\#T_{\a\b}(M)}
\end{equation}
We denote the complex moments of a complex random $X$ by $\E(X^{n_1}\overline X^{n_2})$, $n_1,n_2\in \N\cup \{0\}$. Collecting the results we have proved so far we arrive at the following:
Define 
\begin{equation} m_{n_1,n_2}(v)=\frac{1}{(nl(v))!}\sum_{\substack{u_1,u_2\in V^*\\ l(u_i)=nl(v)}} 
c_{v^{\shuffle n}}(u_1)c_{v^{\shuffle n}}(u_2) 
\prod_{i=1}^{nl(v)}
\inprod{yf_{u_{1,i}}}{yf_{u_{2,i}}}
\end{equation}
if $n_1=n_2=n$ and 0 otherwise!

\begin{thm}\label{thm:finalmoments}For any $\e\neq v\in V^*$
\begin{align} 
     \E((X_M^v)^{n_1}\overline{(X_M^v)}^{n_2})\to & m_{n_1,n_2}(v)\\
    \E((Y_M^v)^{n_1}\overline{(Y_M^v)}^{n_2})\to &  m_{n_1,n_2}(v)
\end{align}
as $M\to\infty.$  

\end{thm}
\begin{proof}
This follows from Theorem \ref{thm:moments} and summation by parts. 
\end{proof}

\section{Limits of complex random variables} 
\label{sec:probability}
In this section we state and prove a Fr\'echet-Shohat-type theorem for complex random variables. This might be known to experts in the field, but we were unable to find a reference in the literature.

A complex random variable $Z$ with probability distribution  $\mu$ is said to be rotationally invariant if \begin{equation}
    \int_\C f(e^{i\theta} z)d\mu(z)=\int_\C f(z)d\mu(z)
\end{equation}
for every $f\in L^{1}(\mu)$ and $\theta\in \R$. Note 
that the density of such a distribution - if it exists - is a function of $\abs{z}$. 

\begin{thm}\label{frechet-shohat}
Let $Z_n$ be complex random variables with probability distribution $\mu_n$. Assume that these all admit finite complex moments
\begin{equation}
    m_{n_1,n_2}^{(n)}:=\int_\C z^{n_1}\overline{z}^{n_2}d\mu_n(z), \quad n_1n_2\in \N_{0}
\end{equation}
Assume that 
\begin{enumerate}
    \item for all $n_1,n_2\in \N_{0}$ we have $m_{n_1,n_2}^{(n)}\to m_{n_1,n_2}\in \C$,
    \item if $m_{n_1,n_2}\neq 0$ then $n_1=n_2$.
\end{enumerate}
Then there exist a rotationally invariant complex random variable $Z$ with complex moment sequence $s=\{m_{n_1,n_2}\}_{n_1, n_2\in \N_{0}}.$

Assume further that 
\begin{enumerate}[resume]
    \item \label{sufficient-convergence}at least one of the marginal random variables $aX+bY$ of $Z$,  $a^2+b^2=1$, (and hence all of them) are uniquely determined by its moments. 
\end{enumerate}
Then $Z_n\to Z$ in distribution as $n\to\infty$, and $\mathbb P(Z_n\in B)\to \mathbb{P}(Z\in B)$ for any $Z$-continuity set  $B\subseteq \C$.
\end{thm}
\begin{proof}
Clearly the set of complex random variables  $Z$ may be examined by considering the corresponding 2-dimensional real random variables $Z'$ via the bijection $\psi:\C\to\R^2$, $z=x+iy\mapsto (x,y)$: $Z'$ has positive Radon measure $\mu'$ on $\R^2$ given by $\mu'(A)=\mu(\psi^{-1}(A))$ for any open set $A\subseteq \R^{2}$.  The complex moments of $Z$ determine and are determined by the real moments 
\begin{equation}s_{m_1,m_2}=\int_{\R^2}x^{m_1}y^{m_2}d\mu'(x,y)
\end{equation}
through $z=x+iy$, $x=(z+\overline z)/2$, $y=(z-\overline z)/2i$. 
We observe that $\mu$ is rotationally invariant if and only if $\mu'$ is rotationally invariant in the sense that for any $(a,b)\in \R^2$ of norm 1 we have, for all open sets $A\in \R^2$ that  $\mu'(RA
)=\mu'(A
)$ where $R=\begin{pmatrix}a &b \\-b &a
\end{pmatrix}$.

The real random variables $Z'_{n}$ have real moments
\begin{equation}
    s_{m_1,m_2}^{(n)}=\int_{\R^2}x^{m_1}y^{m_2}d\mu_{n}'(x,y)
\end{equation}
Since these are finite linear combinations of the complex moments we have that for all $m_1,m_2\in \N_0$
\begin{equation}
    s_{m_1,m_2}^{(n)}\to s_{m_1,m_2}\in \R   
\end{equation}
We wish to show that $s=\{s_{m_1,m_2}\}$ is the moment sequence of a rotationally invariant random variable $Z'$. By a result of Berg and Thill \cite[Prop 2.2]{BergThill:1991} this is the case if i) $s$ is positive definite and ii) the $\C$-linear functional $L_{s}:\C[x,y]\to \C$ given by $L_s(x^{m_1}y^{m_2})=s_{m_1,m_2}$ is rotation invariant in the sense that $L_s(p\circ R)=L_s(p)$ for all polynomials and $R$ as above. 

To see that $s$ is positive definite (meaning that for any finite set $l^{(1)}, \ldots l^{(r)}\in \N_0^{2}$ the real matrix $(s_{l^{(i)}+l^{(j)}})_{i,j=1}^r$ is positive semidefinite) we note that this follows from the fact that $s$ is the limit of $s^{(n)}=\{s^{(n)}_{m_1,m_2}\}$ which is positive definite since it is a moment sequence: That $s^{(n)}$ is positive definite is seen by noticing that for any $\lambda\in \R^r$ we have  
\begin{align}
\lambda^T\left(s^{(n)}_{l^{(i)}+l^{(j)}}\right)_{i,j=1\ldots r}\lambda &=\sum_{i,j=1}^r \lambda_i s^{(n)}_{l^{(i)}+l^{(j)}}\lambda_j
\\
&=\sum_{i,j=1}^r \lambda_i \int_{\R^2} x^{l_1^{(i)}+l_1^{(j)}}y^{l_2^{(i)}+l_2^{(j)}}d\mu'_n(x,y)\lambda_j\\
=&
 \int_{\R^2}\left(\sum_{i}^r \lambda_i x^{l_1^{(i)}}y^{l_2^{(i)}}\right)^2d\mu'_n(x,y)\geq 0.
\end{align}
By letting $n\to\infty$  we see that $\lambda^T\left(s_{l^{(i)}+l^{(j)}}\right)_{i,j=1\ldots r}\lambda\geq 0$ so $s$ is positive definite.

To see that $L_s$ is rotation invariant it suffices to show that $L_s((ax+by)^{m_1}(-bx+ay)^{m_2})$ is independent of $(a, b)\in \R^2$ with norm 1. Note that by linearity we have $L_s(z^{n_1}\overline z^{n_2})=m_{n_1,n_2}$. 
Writing \begin{align}
    ax+by&=\frac{1}{2}(\gamma z+\overline \gamma \overline z)\\
    -bx+ay&=\frac{1}{2i}(\gamma z-\overline \gamma \overline z)
\end{align} with $\g=a-ib$ we find
\begin{align}
    2^{m_1}&(2i)^{m_2}L_s((ax+by)^{m_1}(-bx+ay)^{m_2})=L_s\left(\left(\gamma z+\overline{\gamma}\overline z\right)^{m_1}\left(\gamma z-\overline \gamma \overline z\right)^{m_2}\right)\\
    &=\sum_{\substack{j_1=0\ldots m_1\\
    j_2=0\ldots m_2}}\binom{m_1}{j_1}\binom{m_2}{j_2}\g^{j_1+j_2}\overline{\g}^{m_1+m_2-(j_1+j_2)}(-1)^{m_2-j_2}m_{j_1+j_2,m_1+m_2-(j_1+j_2)}\\
    \intertext{Since $m_{l_1,l_2}=0$ unless $l_1=l_2$ the sum reduces to terms with $j_1+j_2=(m_1+m_2-(j_1+j_2))$ and for such terms $\g^{j_1+j_2}\overline{\g}^{m_1+m_2-(j_1+j_2)}=1$ and we find}
    &=\sum_{\substack{j_1=0\ldots m_1\\
    j_2=0\ldots m_2\\2(j_1+j_2)=(m_1+m_2)}}\binom{m_1}{j_1}\binom{m_2}{j_2}(-1)^{m_2-j_2}m_{j_1+j_2,m_1+m_2-(j_1+j_2)}
\end{align}
which is independent of $(a,b)$. We conclude that $L_s$ is rotation invariant, and \cite[Prop 2.2]{BergThill:1991} allows us to conclude the existence of a 2-dimensional rotationally invariant random variable $Z'$ which has $s_{m_1,m_2}$ as its  moments. It follows that the corresponding rotational invariant complex random variable $Z$ has $m_{n_1,n_2}$ as its complex moments. This proves the first claim of the theorem.

To prove the second claim recall that the Cram\'er-Wold theorem \cite[Thm 29.4]{Billingsley:1995a} states that $Z_n'=(X_n,Y_n)\to Z'=(X,Y)$ in distribution if and only if for any $(a,b)\in \R^2$ with norm 1 we have $aX_n+yY_n\to aX+bY$ in distribution. Since $Z'$ is rotationally invariant 
\begin{equation}
    \int_{R^2}(ax+by)^nd\mu_Z'(x,y)=\int_{R^2}x^nd\mu_Z'(x,y)  
\end{equation}
so all the marginal distributions have the same moments. If one (and therefore all) of the marginal distributions $aX+bY$ has moments which uniquely determine this 1-dimensional distribution then, since by construction 
the moments of $aX_n+bY_n$ converges to the moments of $aX+bY$ and all moments are finite,  the classical Frech\'et-Shohat theorem  \cite[Sec 11.4 C]{Loeve:1977a} gives that $aX_n+bY_n\to aX+bY$ in distribution. By  Cram\'er-Wold this implies that $Z_n'\to Z'$ or equivalently that $Z_n\to Z.$ The last statement follows from the Portmanteau theorem \cite[Thm 2.1]{Billingsley:1999}.
\end{proof}

\begin{rem}In light of Theorem \ref{frechet-shohat} \eqref{sufficient-convergence} it is convenient to have a condition which ensures that the marginal is determined by a given set of moments. The Carleman condition provides this: Consider a complex random variable $Z$ with complex moments \begin{equation}
    m_{m_1,m_1}=\int_\C z^{m_1}\overline z^{m_2}d\mu(z).\end{equation} Then by Cauchy-Schwarz the marginal $(aX+bY)$ has moments bounded as follows
\begin{equation}
    m_{2k}=\int_\C (ax+by)^{2k} d\mu(z)\leq \norm{(a,b)}^{2k}\int_\C {z^k\overline z^k}d\mu(z)=m_{k,k}
\end{equation}
By the classical Carleman condition $aX+bY$ is determined by these moments provided that 
\begin{equation}\sum_{k=1}^\infty\frac{1}{m_{2k}^{\frac{1}{2k}}}=\infty\end{equation}
It follows that if the complex moments satisfies 
\begin{equation}\label{carleman}\sum_{k=1}^\infty\frac{1}{m_{k,k}^{\frac{1}{2k}}}=\infty\end{equation}
then \emph{all} the marginal distributions are determined, and by the Cram\'er-Wold theorem the complex distribution is determined by its moments.
\end{rem}
\begin{rem}\label{rem:uniqueness-condition}
Assume that a complex distribution $Z$ has complex moments satisfying \begin{equation}m_{k,k}\leq C (2k)!\end{equation} for some $C>0$. By Stirling's formula we have  
\begin{equation}m_{k,k}^{\frac{1}{2k}}\leq C^{\frac{1}{2k}} ((2k)!)^{\frac{1}{2k}}\asymp (2\pi2k)^{\frac{1}{4k}}\left(\frac{2k}{e}\right)^{\frac{2k}{2k}}\asymp k\end{equation} so the Carleman condition \eqref{carleman} is satisfied and $Z$ is determined by its moments. 
\end{rem}
\begin{rem}\label{rem:kotz-like}
Consider the Kotz-like distribution function \cite[Eq. 7]{KleiberStoyanov:2013} \begin{equation}\label{kotz-distribution}f\textsubscript{Kotz,$l$}(z)=\frac{(l!)^2}{l\pi}e^{-\abs{l!z}^{\frac2 l}}\abs{l!z}^{2(\frac 1 l-1)}.\end{equation}  The corresponding random variable  is rotationally invariant and the complex moments may be easily computed to be  
\begin{equation}m_{k_1,k_2}=\delta_{k_1=k_2}\frac{(lk_1)!}{(l!)^{2k_1}}.
\end{equation}
For $l=1,2$ the above analysis shows that the corresponding random variable is determined by its moments. For $l=1$ it is the standard complex normal distribution.  For $l>2$ the Carleman condition \eqref{carleman} is not satisfied and we cannot a priori make any conclusion. In fact the Kotz-like distribution is \emph{indetermined} when $l\geq 3$, i.e. it is not determined by its moments as shown in \cite[Cor. 8]{KleiberStoyanov:2013} and there are infinitely many distributions with the same moments. 
\end{rem}

\section{Distributions of noncommutative modular symbols}
\label{sec:distribution}
In the last section we found general conditions to ensure that asymptotic moments determine an asymptotic distribution. In this section we analyze this in the concrete situation of the random variables $X_M^v$ and $Y_M^v$.  

\begin{thm}\label{final-convergence-result}
Let $\e\neq v\in V^*$. 
\begin{enumerate}
    \item 
If $v=f_1dz$ has length 1 with $\norm{f}=1$ then $X_M^v$, $Y_M^v$ converges in distribution to the standard  complex normal distribution with distribution function $f(z)=\frac{1}{\pi}e^{-\abs{z}^2}$.
    
\item\label{final-convergence-length-two} If $v=f_1dzf_2dz$ has length 2 then $X_M^v$, $Y_M^v$ converges in distribution to a radially symmetric distribution $X(v)$. The distribution $X(v)$  depends only on $\{\inprod{yf_i}{yf_j}\}_{i,j=1}^2$.  
 
If $f_1=f_2$ with $\norm{f_i}=1$ then $X(v)$ is the  Kotz-like distribution with distribution function $\frac{1}{\pi\abs{z}}e^{-2\abs{z}}$.
\item If $v$ has length 3 or higher all asymptotic moments of $X_M^v$, $Y_M^v$ exists and there exists at least one radially symmetric distribution with these as its moments.

\end{enumerate}
\end{thm}

\begin{proof}
We start by examining the \lq classical\rq{} case of length 1 i.e. $v=f(z)dz$, with $\norm{f}=1$.  By Theorem \ref{thm:finalmoments} we see that the complex moments of $X_M^v$ and $Y_M^v$ are asymptotically equal to
\begin{align}
    m_{n,m}=\frac{\delta_{m=n}}{n!}\left(\sum_{\substack{u_1\in V^*
    }} c_{v^{\shuffle n}}(u)\right)^2 =n! 
\end{align}
where we have used \eqref{eq:sum-coefficients}. These moments are therefore equal to those of the standard complex normal distribution function $\frac{1}{\pi}e^{-\abs{z}^2}$. Remark \ref{rem:uniqueness-condition} gives that this distribution is determined by its complex moments, and hence Theorem \ref{frechet-shohat} gives that $X_M^v$ and $Y_M^v$ converges in distribution to the standard normal complex distribution. After adjusting for different normalizations this recovers \cite[Thm 5.1]{PetridisRisager:2004a} and
\cite[Cor 7.8]{PetridisRisager:2018a} with $I=\R/\Z$. Hence we have also reproved Mazur and Rubins conjecture on the normal distribution of modular symbols.  

Moving to $v=f_1dzf_2dz$ of length two we find from Theorem \ref{thm:finalmoments} that the complex moments of $X_M^v$ and $Y_M^v$ are asymptotically equal to
\begin{align} m_{n_1,n_2}(v)&=\frac{\delta_{n_1=n_2}}{(2n_1)!}\sum_{\substack{u_1,u_2\in V^*\\ l(u_i)=2n_1}} 
c_{v^{\shuffle n_1}}(u_1)c_{v^{\shuffle n_1}}(u_2) 
\prod_{i=1}^{2n_1}
\inprod{yf_{u_{1,i}}}{yf_{u_{2,i}}}\\
&\leq \max{(\norm{f_1}^2,\norm{f_2}^2)}{\frac{(2n_1)!}{(2!)^{2n_1}}}\leq C_{f_1,f_2} (2n_1)!
\end{align}
Here we have used that $u_i$ has to be a word in $f_1dz$, $f_2dz$ in order for $c_{v^{\shuffle n_1}}(u_1)$ to be non-zero.

Remark \ref{rem:uniqueness-condition} gives that these moments determines a unique distribution $X(v)$, and since the moments only depend on $\{\inprod{yf_i}{yf_j}\}_{i,j=1}^2$ so does the distribution. We may conclude from Theorem \ref{thm:finalmoments} that $X_M^v$ and $Y_M^v$ converges to $X(v)$ in distribution.

In the special case $v=f_1dzf_2dz$ where $f_1=f_2$ with $\norm{f_i}=1$ we find using \eqref{eq:sum-coefficients} that 
\begin{equation}
    m_{n_1,n_2}(v)=\delta_{n_1=n_2}\frac{(2n_1)!}{(2!)^{2n_1}}.
\end{equation} 
Noticing that the Kotz-like distribution with distribution function 
\begin{equation}\frac{1}{\pi}e^{-2\abs{z}}\abs{z}^{-1}\end{equation}
is the unique distribution with these moments proves the claim.

Finally we move to the length $l\geq 3$ case, where Theorem \ref{thm:finalmoments} gives that all the complex asymptotic moments of $X_M^v$ and $Y_M^v$ exist and that there exist a rotationally invariant distribution with these as its moments. 
\end{proof}

\begin{rem} 
If we consider the \lq trivial \rq{} length $l$ case $v=(f(z)dz)^l$ with $\norm{f}=1$ we see using \eqref{eq:sum-coefficients} that
the asymptotic moments of $X_M^v$ and $Y_M^v$ equal 
\begin{equation}
    m_{n_1,n_2}(v)=\delta_{n_1=n_2}\frac{(ln_1)!}{(l!)^{2n_1}}.
\end{equation}
This equals the Kotz-like distribution from Remark \ref{rem:kotz-like} which is indetermined when $l\geq 3$. This is probably related to the fact that if $X$ has normal distribution then $X^n$ is indeterminate for $n\geq 3$ (See \cite{Berg:1988}). 

Since the Kotz-like distribution is indetermined for $l\geq$ Theorem \ref{frechet-shohat} \eqref{sufficient-convergence} does not allow us to conclude convergence in distribution. However using the  Portmanteau theorem (\cite[Thm 2.1]{Billingsley:1999}) we see that if $Z_M$ converges in distribution to a random variable $Z$ then for any continuous function $f$ on $\C$ we have that $f(Z_M)$ converges in distribution to $f(Z)$. In particular we find, using \eqref{powers} that when $v=(f(z)dz)^l$ we have that $X_M^v$ and $Y_M^v$ converges in distribution to $\frac{Z^l}{l!}$ where $Z$ is the standard normal distribution on $\C$. Note that 
\begin{align}\mathbb E(g(\frac{Z^l}{l!}))&=\int_{\R^2}{g(\frac{z^l}{l!})}\frac{e^{-\abs{z}^2}}{\pi}dxdy\\
&=\int_{0}^{2\pi}\int_0^\infty g(\frac{r^l}{l!}e^{il\theta})\frac{e^{-r^2}}{\pi}rdrd\theta \\
&=\int_{0}^{2\pi}\int_0^\infty g(re^{i\theta)}\frac{(l!)^2}{\pi l}\frac{e^{-(l!r)^{2/l}}}{(l!r)^{2(1-1/l)}} rdrd\theta=\int_{\R^2}g(z)f\textsubscript{Kotz,$l$}(z)dxdy
\end{align}
\end{rem}
 It follows that in this case $X_M^v$ and $Y_M^v$ indeed converges in distribution to Kotz-like distribution with distribution function $f\textsubscript{Kotz,$l$}(z)$ \eqref{kotz-distribution}.

\begin{rem}
More generally, if we assume $\inprod{yf_{v_i}}{yf_{v_j}}\geq 0$ and $\norm{f_{v_i}}=1$ we find from \eqref{eq:lower-bounds-coefficients}  the following lower bound
\begin{align}
    m_{n_1,n_2}(v)=\sum_{\substack{u_1,u_2\in V^*\\ l(u_i)=nl(v)}}& c_{v^{\shuffle n}}(u_1)c_{v^{\shuffle n}}(u_2) \prod_{i=1}^{nl(v)}\inprod{yf_{u_{1,i}}}{yf_{u_{2,i}}}\\ &\geq[c_{v^{\shuffle n}}((f_{v_1}(z)dz)^n)\cdots(f_{v_l(v)}(z)dz)^n)]^2=((n!)^{l(v)})^2, 
\end{align}
and we find 
\begin{equation}
    m_{n,n}(v)\geq \frac{(n!)^{2l(v)}}{(l(v)n)!}\geq cn^{l/2}
\end{equation}
where we have used Stirlings formula for the factorials. This shows that if $l\geq 3$ then the Carleman condition \eqref{carleman} is not satisfied if $\inprod{yf_{v_i}}{yf_{v_j}}\geq 0$ and $\norm{f_{v_i}}=1$. So under these assumptions $X(v)$ may be indetermined.
\end{rem}

\begin{example}\label{fun-example}
The distributions $X(v)$ in Theorem \ref{final-convergence-result}, \eqref{final-convergence-length-two} can in some cases be described explicitly. For instance if we choose two \emph{orthonormal} weight 2 cusp forms $f_1$, $f_{-1}$ and let $v=v_1v_{-1}=f_1dzf_{-1}dz$, then Theorems \ref{thm:finalmoments} and \ref{final-convergence-result} give that $X(v)$ is the unique distribution with moments
\begin{equation}m_{n_1,n_2}(v)=\frac{\delta_{n_1=n_2}}{(2n_1)!}\sum_{\substack{u_1 \in V^*}} 
c_{v^{\shuffle n_1}}^2(u_1).\end{equation}
We claim that for $v=v_1v_{-1}$ we have the combinatorial identity
\begin{equation}\label{combinatorial-identity}
    \sum_{\substack{u \in V^*}} 
c_{v^{\shuffle n}}^2(u)=n!^2\left(\frac{1}{\cos(x)}\right)^{(2n)}(0).
\end{equation}
Therefore the complex diagonal moments are given by
\begin{equation}\label{curious-observation}
    m_{n,n}(v)=\binom{2n}{n}^{-1}\left(\frac{1}{\cos(x)}\right)^{(2n)}(0)
\end{equation}

The distribution function $f(z)=h(\abs{z})$ of $X(v)$ satisfies \begin{equation}
\int_{-\infty}^{\infty}h(\abs{r})\abs{r}r^{n}dr=\begin{cases}\frac{1}{\pi}m_{l,l}&\textrm{if }n=2l\\
0&\textrm{otherwise.}
\end{cases}
\end{equation}
such that $\tilde h(r)=h(\abs{r})\abs{r}$ has (two-sided) Laplace transform \begin{equation}\mathscr L (\tilde h)(s)=\frac{1}{\pi}\sum_{n=0}^\infty \frac{s^{2n}}{(2n)!}m_{n,n}
\end{equation}
We can now use $\binom{2n}{n}^{-1}=(2n+1)\int_{0}^1(y(1-y))^{n}dy$ (which follows from the well-known properties of the beta function) to see that  
\begin{equation}\mathscr L (\tilde h)(s)=\frac{1}{\pi}\int_{0}^1g(s\sqrt{y(1-y)})dy
\end{equation} where $g(s)=\frac{d}{ds}(\frac{s}{cos(s)}).$ Applying the inverse Laplace transform and $\mathscr L (\frac{1}{\cosh(r)})(s)=\frac{\pi}{\cos(\pi s/2)}$ we find
\begin{equation}
    h(r)=\frac{1}{4}\int_{0}^1\frac{1}{y(1-y)}\frac{\displaystyle\sinh\left(\frac{\pi r}{2\sqrt{y(1-y)}}\right)}{\cosh^2\left(\displaystyle\frac{\pi r}{2\sqrt{y(1-y)}}\right)}dy
\end{equation} for $r\geq 0$. So if \eqref{combinatorial-identity} holds as claimed, then 
$X(v)$ has distribution function $f(z)=h(\abs{z})$.

We verify \eqref{combinatorial-identity} as follows: 

Fra\c{c}on and Viennot \cite{FranconViennot:1979a} were able to precisely count the number of permutations of $[m]=\{1,2, \ldots, m\}$ of various types, where a type is given by the sets of values of  peaks, valleys, double rises and double falls. This allowed them to deduce various identities including \cite[Cor 4.2]{FranconViennot:1979a}
\begin{equation}\label{Dumont-identity}
    \left(\frac{1}{\cos(x)}\right)^{(2n)}(0)=\sum_{\g\in \mathscr{C}_{2n+1}'}S(\g).
\end{equation}
Here $\mathscr{C}_{m}'$ is the set of strict contractions of length $m$, i.e. maps $\g:[m]\to[m]$  satisfying \begin{enumerate}
\item $\g(1)=\g(m)=1$,
\item $\abs{\g(i+1)-\g(i)}=1\textrm{ for }i=1, \ldots m-1$.
\end{enumerate}
and $S$ is defined by \begin{equation}
    S(\g)=\prod_{i=1}^{m-1      }\min (\g(i), \g(i+1)).
\end{equation}

We will show that \eqref{combinatorial-identity} follows from \eqref{Dumont-identity}.
Recall that the set of length $2n$ ballot sequences $B_{2n}$ is the set of  $2n$-tuples $b=(b_1,b_2,\ldots,b_{2n})$ consisting of $n$ occurrences of 1 and $n$ occurrences of $-1$ and satisfying $\sum_{i\leq j}b_i\geq 0$ for $1\leq j\leq 2n$. We have a bijection between the set of strict contractions of odd length and the set of ballot sequences given by
\begin{equation}
\begin{array}{ccc} 
B_{2n}&\to& \mathscr{C}_{2n+1}'\\
 b &\mapsto & \g(i)=1+\sum_{j\leq i-1}b_j
\end{array}
\end{equation}
with inverse $\g\mapsto (\g(2)-\g(1), \ldots ,\g(2n+1)-\gamma(2n))$. This translates  \eqref{Dumont-identity} to 
\begin{equation}\label{massaged-Dumont}
   \left(\frac{1}{\cos(x)}\right)^{(2n)}(0)= \sum_{b\in B_n}\prod_{\substack{i=1\\b_i=-1}}^{2n}(1+\sum_{j\leq i}b_j)\prod_{\substack{i=1\\b_i=
    1}}^{2n}(1+\sum_{j\leq i-1}b_j)
\end{equation}
We now relate this to the shuffle coefficients $c_{v^{\shuffle n}}^2(u)$. The $2n$-words obtainable under $n$ shuffles of $v=v_1v_{-1}$ are precisely $v_b:=v_{b_1}v_{b_2}\cdots v_{b_{2n}}$ for $b\in B_{2n}$. So 
\begin{equation}
    v^{\shuffle n}=\sum_{b\in B_{2n}}c_{v^{\shuffle n}}(v_b)v_b.
\end{equation}
To compute $c_{v^{\shuffle n}}(v_b)$ we may first place the $n$ occurrences of $v_1$ at the $i$th positions with $b_i=1$ . This may be done in $n!$ ways. We then place (from left to right) the $n$ occurrences of $v_{-1}$. If some $v_{-1}$ should be placed on the $i$th position with $b_i=-1$ then there are $\sum_{j\leq i-1}b_j=1+\sum_{j\leq i}b_j$ ways of doing so. It follows that 
\begin{equation}
    c_{v^{\shuffle n}}(v_b)=n!\prod_{\substack{i=1\\b_i=-1}}^{2n}(1+\sum_{j\leq i}b_j).
\end{equation}
Alternatively we may first place the $n$ occurences of $v_{-1}$ at the $i$th positions with $b_i=-1$ which may be done in $n!$ ways. We then place (from right to left) the $n$ occurrences of $v_{1}$. If $v_1$ should be placed on the $i$th position with $b_i=1$ then there are $-\sum_{j> i}b_j=\sum_{j\leq  i}b_j=1+\sum_{j\leq  i-1}b_j$ ways of doing this. It follows that 
\begin{equation}
    c_{v^{\shuffle n}}(v_b)=n!\prod_{\substack{i=1\\b_i=1}}^{2n}(1+\sum_{j\leq i-1}b_j).
\end{equation}
Comparing these expressions we find that the right-hand-side of \eqref{massaged-Dumont} equals \begin{equation}\sum_{b\in B_{2n}} (c_{v^{\shuffle n}}(v_b)/n!)^2\end{equation} which proves \eqref{combinatorial-identity}.
\end{example}

\section{Special values of multiple \texorpdfstring{$L$}{L}-functions}
\label{sec:L-series}
In this section we explain how the noncommutative modular symbols 
$I_{i\infty}^{a/c}(v)$ 
may be interpreted as special values of multiple $L$-functions twisted by additive characters. For the case of ordinary modular symbols a similar analysis may be found in \cite[Sect. 3.3]{Nordentoft:2021c} where the multiple $L$-function is simply a standard Hecke $L$-function.  The additive twists for multiple $L$-functions which we develop may also be seen as $GL_2$ analogs of multiple polylogarithms (see e.g. \cite{Waldschmidt:2002}). 

Related analyses also using iterated Mellin transforms may be found in \cite{Manin:2006, ChoieIhara:2013, Brown:2019b}

\subsection{Continuation and functional equation of multiple \texorpdfstring{$L$}{L}-functions}
In this section we work with $\G=\G_0(q)$ such that both $0$ and $i\infty$ are cusps for $\Gamma$. For  $v=w_1w_2\ldots w_l$ with $w_i(z_i)=f_i(z_i)dz_i$ holomomorphic cuspidal 1-forms on $\GmodH$ we consider, for $s=(s_1,s_2,\ldots,s_l)$, the function 
\begin{equation}\label{integral-representation}
    I_{i\infty}^0(v,s)=\int_{i\infty}^0f_1(z_1)z_1^{s_1}\int_{i\infty}^{z_1}f_2(z_2)z_2^{s_2}\ldots\int_{i\infty}^{z_{l-1}}f_l(z_l)z_l^{s_l}\frac{dz_l}{z_l} \ldots \frac{dz_1}{z_1}
\end{equation}
We note that this since both $0$ and $i\infty$ are cusps for $\G$ the integral converges uniformly and absolutely for $s_i$ in any compac subset of $\C$ as long as we choose the line integral to only intersect finitely many fundamental domains. We note also that $I_{i\infty}^0(v)=I_{i\infty}^0(v,1)$ where 1 should be interpreted as the vector in $\C^n$ with each entry equal to 1.

In order to not make the notation to heavy we restrict in this section to length $l=2$, so that $v(z)=f_1(z_1)dz_1f_2(z_2)dz_2$ but it is straightforward to generalize all statements to general length. Almost all of the statements also have generalizations to general even weight $k$ cusp forms.

Denote the Fourier coefficients of $f_i$ by 
\begin{equation}
    f_i(z)=\sum_{n=1}^\infty a_i(n)e(nz),
\end{equation}
We consider the series
\begin{equation}\label{series-representation}
    L(f_1,f_2, s_1, s_2)=\sum_{n_1,n_2=1}^{\infty}\frac{a_1(n_1)a_2(n_2)}{(n_1+n_2)^{s_1}n_2^{s_2}}.
\end{equation}
By well-known bounds for the Fourier coefficients this sum converges absolutely for $\Re(s_1)> 2$, $\Re(s_2)\geq 1$. For $s_2=1$ this function may be analyzed using iterated integrals as follows (See also \cite[Sec. 2.1.3]{Manin:2006}\cite[Sec3.3]{Manin:2009}).

Inserting the Fourier expansions of $f_1$ and $f_2$ and setting $s_2=1$ we find 
\begin{equation}\label{integral-expression}
    I_{i\infty}^0(v,s_1,1)
    =\frac{i^{s_1+1}}{(2\pi)^{s_1+1}}\Gamma(s_1)L(f_1,f_2,s_1,1)
    \end{equation}
 There are similar but slightly more complicated expressions when $s_2$  is any positive integer. These can be proved by using formulas for the incomplete Gamma function. 
 
 Recall that the Fricke involution 
 \begin{equation}Wf_j(z)=\frac{1}{qz^2}f_j(-1/(qz))\end{equation} maps the space of weight $2$ cusp forms of $\Gamma_0(q)$ to itself. We write \begin{equation}Wv=Ww_1(z_1)Ww_2(z_2)\end{equation} where $Ww_i(z_i)=(Wf_i)(z_i)dz_i$. Writing \begin{equation}
     f_i(z_i)=(qz_i^2)^{-1}Wf_i(-1/(qz_i))
 \end{equation} and making a simple change of variables it is straightforward to verify the functional equation
 \begin{align}
     I_{i\infty}^0(v,s)=&(-1)^{1-{s_1+s_2}}q^{2-{(s_1+s_1)}}\\
     &\left(I_{i\infty}^0(Wv,2-s)-I_{i\infty}^0((Wf_1)(z)dz,2-s_1)I_{i\infty}^0((Wf_1)(z)dz,2-s_2)\right).
 \end{align}
 Note that in the length 1 case we have 
 \begin{align}
     I_{i\infty}^0(gdz,s)&=\int_{i\infty}^0 g(z)z^{s-1}dz\\ &=-i^s\frac{\Gamma(s)}{(2\pi)^s}L(g,s)=(-q)^{2-s}I_{i\infty}^0(Wg(z)dz ,2-s),
 \end{align}
where \begin{equation}\label{series-representation-II}L(g,s)=\sum_{n=1}^\infty\frac{a_g(n)}{n^s}.\end{equation} 
Recall that if $f_i$ are Hecke forms then we have $Wf_i=\epsilon_i \hat f_i$ with $\epsilon_i$ on the unit circle and $\hat f=\overline f(-\overline z)$ (which corresponds to conjugating the Fourier coefficients). See \cite[Sec 14.7]{IwaniecKowalski:2004a}.

Summarizing and specializing to $s_2=1$ we get the following theorem which extends a classical result due to Hecke to multiple $L$-functions:
\begin{thm}\label{thm:functional-eq-multiple-zeta} The functions $L(f_1,f_2, s_1, 1)$, $L(\hat f_1,\hat f_2, s_1, 1)$, $L(f_i, s_i)$, $L(\hat f_i, s_i)$ defined for $\Re(s_1)>2$ by \eqref{series-representation} and \eqref{series-representation-II} admit analytic continuation to $s\in \C$. If $f_1,f_2$ are Hecke eigenforms with $Wf_i=\epsilon_i \hat f_i$ the following functional equation holds: Let 
\begin{align}
    \Lambda(f_1,f_2,s)&=\left(\frac{q^{1/2}}{2\pi}\right)^{s+1}\Gamma(s)L(f_1,f_2,s,1)\\
    \Lambda(f_i,s)&=\left(\frac{q^{1/2}}{2\pi}\right)^{s}\Gamma(s)L(f_i,s)\\
\end{align} and similar for $\hat f_1$, $\hat f_2$. Then
\begin{align}
    \Lambda(f_i,s)&=-\epsilon_i\Lambda(\hat f_1,2-s),\\
    \Lambda(f_1,f_2,s)&=-\epsilon_1\epsilon_2\left(\Lambda(\hat f_1,\hat f_2,2-s)-\Lambda(\hat f_1,2-s)\Lambda(\hat f_2,1)\right).
\end{align}
\end{thm}

\subsection{Additive twists of multiple \texorpdfstring{$L$}{L}-functions.}
In the previous section we analyzed the iterated integral $I_{i\infty}^0(v,s_1,1)$ and saw how it relates to $L(f_1,f_2,s,1)$. For the groups $\Gamma_0(q)$ with $q>1$ the cusps at $i\infty$ and $0$ are inequivalent. We will now see how for equivalent cusps $i\infty$ and $a/c$ the iterated integral $I_{i\infty}^{a/c}(v)$ may be interpreted as the function  $L(f_1,f_2,s,1)$ twisted by the additive character $n\mapsto e(\frac{a}{c}n)$ evaluated the central point $s=1$.

Consider a cusp $a/c$ equivalent to the cusp at infinity. i.e. $a/c=\gamma(i\infty)$ for some $\g\in \G$. Analogous to  \eqref{integral-representation} we define the function $I_{i\infty}^{\frac{a}{c}}(v,s)$ by
\begin{align}\label{integral-representation-general}
\int_{i\infty}^{a/c}\!\!\!\!\!f_1(z_1)(z_1-\frac{a}{c})^{s_1-1}\!\!\!\int_{i\infty}^{z_1}\!\!\!\!\!f_2(z_2)(z_2-\frac{a}{c})^{s_2-1}\!\!\!\!\!\ldots\int_{i\infty}^{z_{l-1}}\!\!\!\!\!f_l(z_l)(z_1-\frac{a}{c})^{s_l-1}dz_l\ldots dz_1
\end{align}
and notice that this converges uniformly and absolutely for $s_i$ in any compact subset of $\C$, and that $I_{i\infty}^{a/c}(v)=I_{i\infty}^{a/c}(v,1)$.

Plugging in the Fourier expansion we find that for $\Re(s_1)$, $\Re(s)>2$
\begin{align}
\label{intrep-l1}I_{i\infty}^{\frac{a}{c}}(f(z)dz,s)&=-\frac{i^s}{(2\pi)^s}\Gamma(s)L(f,a/c,s)\\
\label{intrep-l2}I_{i\infty}^{\frac{a}{c}}(f_1dz_1f_2dz_2,s_1,1)&=\frac{i^{s_1+1}}{(2\pi)^{s_1+1}}\Gamma(s_1)L(f_1,f_2,a/c,s_1,1)
\end{align}
Here \begin{align}
    L(f,a/c,s)&=\sum_{n=1}^\infty\frac{a_f(n)e(\frac{a}{c}n)}{n^s}\\
    L(f_1,f_2,a/c,s_1,1)&=\sum_{n_1,n_2=1}^\infty\frac{a_1(n_1)a_2(n_2)e(\frac{a}{c}(n_1+n_2))}{n_2(n_1+n_2)^{s_1}}
\end{align}
are the length 1 and 2 $L$-functions twisted by additive characters. These are absolutely convergent for $\Re(s)>2$. 

By using $\gamma(-d/c+i/(cy))=a/c+iy/c$, $j(\gamma, -d/c+i/(cy))=)i/y$, the automorphicity of $f_i$, and a change of variable we find that
\begin{align*}
    I_{i\infty}^{a/c}(f(z)dz,s) &=(-1)^sc^{2-2s}I_{i\infty}^{-d/c}(f(z)dz,2-s)\\
    I_{i\infty}^{a/c}(v,s)=&(-1)^{s_1+s_2}c^{2(2-(s_1+s_2))}\\ & \cdot(I_{i\infty}^{-d/c}(f_1dz,2-s_1)I_{i\infty}^{-d/c}(f_2dz,2-s_2)-I_{i\infty}^{-d/c}(v,2-s))
\end{align*}
where $v=f_1dz_1f_2dz_2$. From these observations we arrive, after a small computation, at the following theorem.
\begin{thm}\label{thm:functional-eq-multiple-zeta-twisted} For each cusp $a/c=\gamma(i\infty)$ with $\g\in \G$ the series $L(f,a/c,s)$, $L(f_1,f_2,a/c,s_1,1)$ admit meromorphic continuation to $s,s_1 \in \C$. The completed $L$-functions
\begin{align}
    \Lambda_{\frac{a}{c}}(f,s)&=\left(\frac{c}{2\pi}\right)^{s}\Gamma(s)L(f, \frac{a}{c},s), \\
    \Lambda_{\frac{a}{c}}(f_1,f_2,s)&=\left(\frac{c}{2\pi}\right)^{s+1}\Gamma(s)L(f_1,f_2,\frac{a}{c} ,s,1)
\end{align}
satisfies the following functional equations
\begin{align}
    \Lambda_{\frac{a}{c}}(f,s)&=-\Lambda_{-\frac{d}{c}}(f,2-s) \\
    \Lambda_{\frac{a}{c}}(f_1,f_2,s)&=-(\Lambda_{-\frac{d}{c}}(f_1,f_2,2-s)-\Lambda_{-\frac{d}{c}}(f_1,2-s)\Lambda_{-\frac{d}{c}}(f_2,1)).
\end{align}
Here $-\frac{d}{c}=\gamma^{-1}(i\infty)$.
\end{thm}
\begin{rem} Note that by \eqref{intrep-l1} and \eqref{intrep-l2} we have 
\begin{align}
    L(f,a/c,1)&=2\pi i I_{i\infty}^{a/c}(f(z)dz)\\
    L(f_1,f_2,a/c,1)&=(2\pi i)^2I_{i\infty}^{a/c}(f_1dz_1f_2dz_2)
\end{align}
More generally one finds that 
\begin{align}
    L(f_1,\ldots,f_l,a/c,s_1, 1_{l-1})&=\sum_{n_1,\ldots, n_l=1}^{\infty}\frac{a_1(n_1)\cdots a_l(n_l)e(\frac{a}{c}(n_1+\cdots n_l))}{(n_1+\cdots+ n_l)^{s_1}(n_2+\cdots+ n_l)\cdots(n_{l-1}+n_l)n_l}\\
&=\sum_{0=m_{l+1}<m_l<\ldots <m_1   }\frac{a_1(m_1-m_2)\cdots a_l(m_l-m_{l+1})e(\frac{a}{c}m_1)}{m_1^{s_1}m_2\cdots m_l}
\end{align}
converges for $\Re(s_1)>1+l/2$ and admits analytic continuation to $s\in\C$. This continuation has central value
\begin{equation}
    L(f_1,\ldots,f_l,a/c, 1)=(2\pi i)^lI_{i\infty}^{a/c}(f_1dz_1\ldots f_ldz_l)
\end{equation}
where we write  $1$ instead of $1_l$.
In combination with  \eqref{shuffle-product} this shows that these central values surprisingly satisfy a shuffle product formula, namely
\begin{align}
L(f_{i_1}\ldots f_{i_m},a/c,1)&L(f_{i_{m+1}}\ldots f_{i_{m+n}},a/c,1)\\& =\sum_{\sigma \in \Sigma_{m,n}}L(f_{i_{\sigma(1)}}\ldots f_{i_{\sigma(m+n)}},a/c,1).
\end{align}

\end{rem}
\section{Numerics}
\label{sec:numerics}
When we want to compute numerically  $I_{i\infty}^{a/c}(v)$ where $\frac a c=\g(i\infty)$ the following considerations are useful: 

Let $z_0\in \H$. By Proposition \ref{prop:symbols-properties} \ref{decomposition-paths}, \ref{gamma-invariance}, and the reversal of path formula \cite[Thm 3.19]{GilFresan:2017}  we find that in the lenght 1 case
\begin{align}
    I_{i\infty}^{a/c}(v_1)&=-I_{i\infty}^{\g^{-1}z_0}(v_1)+I_{i\infty}^{z_0}(v_1)\\
\intertext{and in the length 2 case}
 I_{i\infty}^{a/c}(v_1v_2)&=I_{i\infty}^{\g^{-1}z_0}(v_2v_1)-I_{i\infty}^{\g^{-1}z_0}(v_1)I_{i\infty}^{z_0}(v_2)+I_{i\infty}^{z_0}(v_1v_2)
\end{align}
Using the $q$-expansion of the cusp forms these integrals converge faster if $\g^{-1}z_0$ and $z_0$ has as large an imaginary part as possible. Balancing the two imaginary parts we find that $\abs{j(\g^{-1},z_0)}=1$ with $\Im(z_0)$ as large as possible is optimal, so we  choose \begin{equation}
    z_0=\frac a c+\frac{i}{\abs{c}}.
\end{equation}  Hence we will get good accuracy if we  truncate the $q$-expansion at a height $N$ where $N/\abs{c}$ is not small. 

Since we are interested in letting $c$ grow to infinity it is convenient to compute the integrals only for a set of generators of the group $G$, as these has a bounded value of $c$. We then use the following relations to compute $I_{i\infty}^{\g i\infty}(v_1v_2)$ for a general $g$. 

By repeated use of Proposition \ref{prop:symbols-properties} \ref{decomposition-paths}, \ref{gamma-invariance} we see that 
\begin{equation}
    I_{i\infty}^{\g_1\ldots \g_L i\infty}(v_1v_2)=\sum_{j=1}^L I_{i\infty}^{\g_j i\infty}(v_1v_2)+\sum_{1\leq j\leq l\leq L}I_{i\infty}^{\g_l i\infty}(v_1)I_{i\infty}^{\g_j i\infty}(v_2)
\end{equation}
We note that the computational error on the values of  $I_{i\infty}^{\g_j i\infty}(v_i)$ accumulates quadratically in the word length, whereas the computational error on $I_{i\infty}^{\g_j i\infty}(v_1v_2)$ accumulates only linearly in the word length. 

Using the same technique one finds that for all $n_j\in \Z$ we have 
\begin{align}
    I_{i\infty}^{\g_1^{n_1}\ldots \g_L^{n_L} i\infty}(v_1v_2)&=\sum_{j=1}^L n_jI_{i\infty}^{\g_j i\infty}(v_1v_2) +\frac{n_j(n_j-1)}{2}I_{i\infty}^{\g_j i\infty}(v_1)I_{i\infty}^{\g_j i\infty}(v_2)\\&\quad +\sum_{1\leq j\leq l\leq L}n_jn_lI_{i\infty}^{\g_l i\infty}(v_1)I_{i\infty}^{\g_j i\infty}(v_2).
\end{align}

Computer programs such as SAGE has good algorithms which uses Farey symbols to find generators for congruence groups and to solve the the word problem in this context. These are based on \cite{Kulkarni:1991}  (See also \cite{KurthLong:2008}). Using the above techniques we plot various histograms of $(\frac{\vol{\GmodH}}{4\log(c^2)})^{l/2}(2\pi i)^l I_{i\infty}^{a/c}(v)$ with logarithmic colorgrading. For the first 5 plots we use $q=37$ and $M=20000$, which gives about 3 million datapoints. For the final plot we have used $q=41$ and $M=5500$ which gives about $2.2\cdot 10^5$ datapoints. In all the plots we observe the radial symmetry predicted by Theorem \ref{final-convergence-result}.

\begin{figure}[H]
    \centering
    \begin{subfigure}[b]{0.43\textwidth}
        \includegraphics[width=\textwidth]{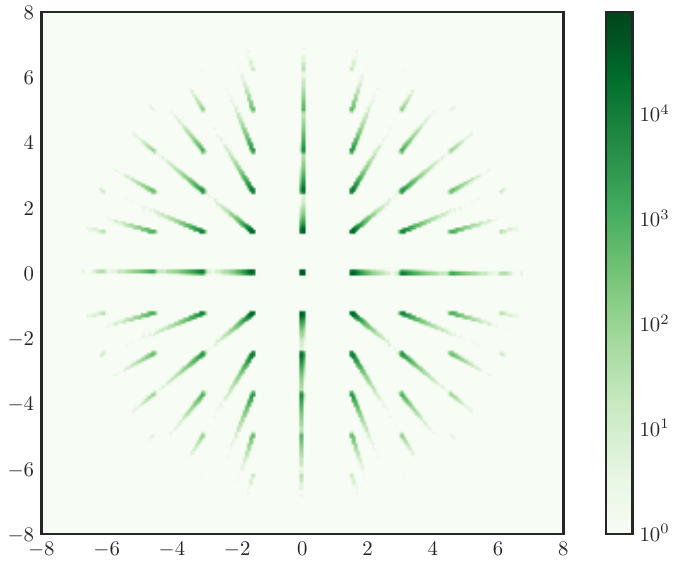}
        \caption{length 1}
    \end{subfigure}
    ~ 
    \begin{subfigure}[b]{0.43\textwidth}
        \includegraphics[width=\textwidth]{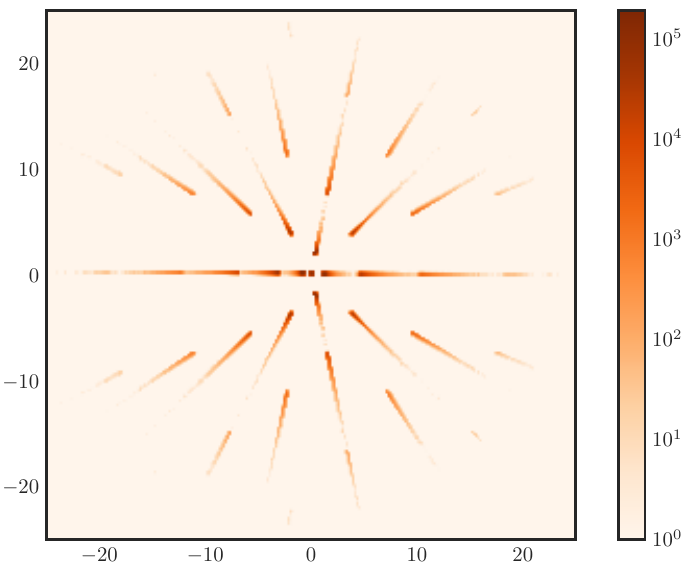}
        \caption{length 2, same form}
    \end{subfigure}
\end{figure}
\begin{figure}[H]
    \centering
    \begin{subfigure}[b]{0.43\textwidth}
        \includegraphics[width=\textwidth]{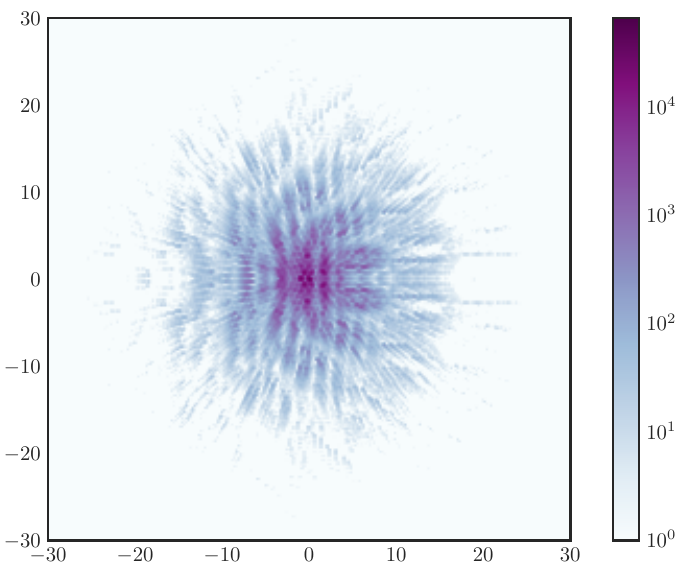}
        \caption{length 2, different non-orth. forms}
    \end{subfigure}
    ~ 
    \begin{subfigure}[b]{0.43\textwidth}
        \includegraphics[width=\textwidth]{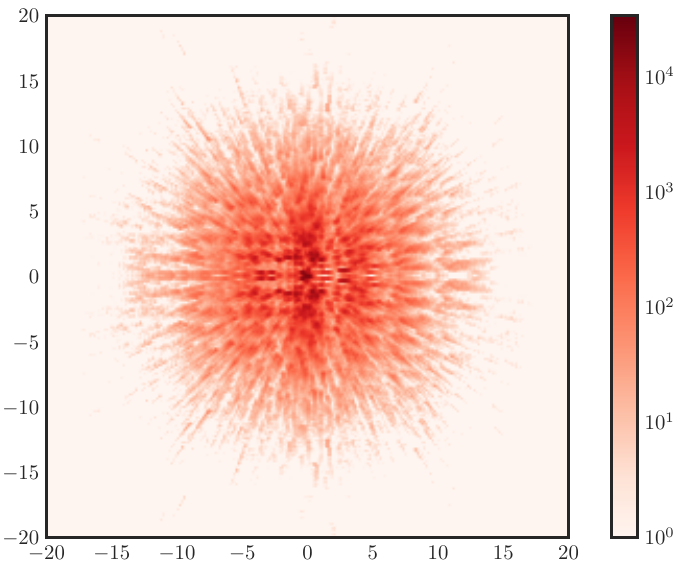}
        \caption{length 2, orthogonal forms}
    \end{subfigure}
\end{figure}
\begin{figure}[H]
    \centering
    \begin{subfigure}[b]{0.43\textwidth}
        \includegraphics[width=\textwidth]{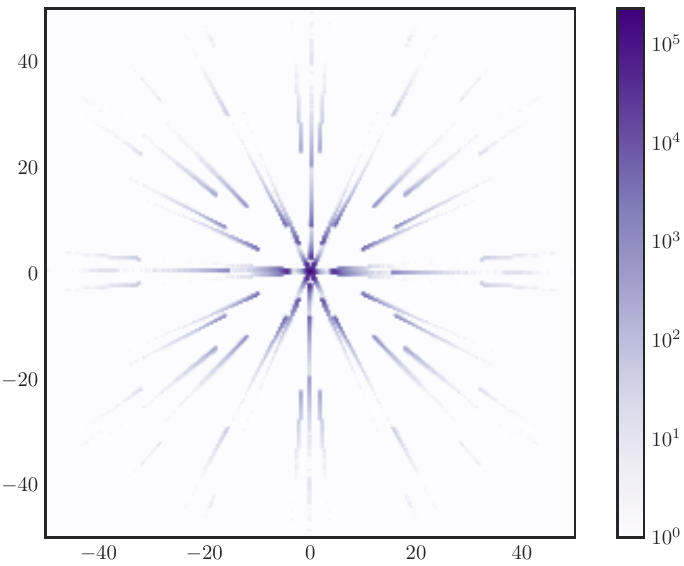}
        \caption{length 3, same forms}
    \end{subfigure}
    ~ 
    \begin{subfigure}[b]{0.43\textwidth}
        \includegraphics[width=\textwidth]{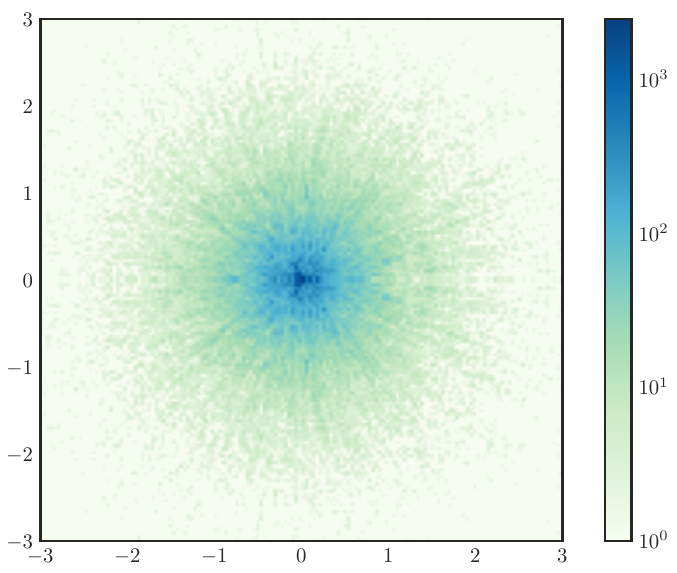}
        \caption{length 3, different forms}
    \end{subfigure}
\end{figure}

\section*{Acknowledgments} We are grateful to Christian Berg for many useful comments on Section \ref{sec:probability}, and to Wadim Zudilin, Christian Krattenthaler and Guoniu Han for discussions and correspondence that led to \eqref{combinatorial-identity}.

\bibliographystyle{amsplain}

\begin{thebibliography}{10}

\bibitem{BanksPanzerPym:2020a}
Peter Banks, Erik Panzer, and Brent Pym, \emph{Multiple zeta values in
  deformation quantization}, Inventiones mathematicae \textbf{222} (2020),
  no.~1, 79--159.

\bibitem{Berg:1988}
Christian Berg, \emph{The {{Cube}} of a {{Normal Distribution}} is
  {{Indeterminate}}}, The Annals of Probability \textbf{16} (1988), no.~2,
  910--913.

\bibitem{BergThill:1991}
Christian Berg and Marco Thill, \emph{Rotation invariant moment problems}, Acta
  Mathematica \textbf{167} (1991), 207--227.

\bibitem{Bettin:2019}
Sandro Bettin, \emph{High moments of the {{Estermann}} function}, Algebra \&
  Number Theory \textbf{13} (2019), no.~2, 251--300.

\bibitem{BettinDrappeau:2019}
Sandro Bettin and Sary Drappeau, \emph{Limit laws for rational continued
  fractions and value distribution of quantum modular forms}, arXiv:1903.00457
  (2019), 42 pp.

\bibitem{Billingsley:1995a}
Patrick Billingsley, \emph{Probability and measure}, third ed., Wiley
  {{Series}} in {{Probability}} and {{Mathematical Statistics}}, {John Wiley \&
  Sons Inc.}, {New York}, 1995. \MR{1324786}

\bibitem{Billingsley:1999}
\bysame, \emph{Convergence of probability measures}, second ed., Wiley
  {{Series}} in {{Probability}} and {{Statistics}}: {{Probability}} and
  {{Statistics}}, {John Wiley \& Sons, Inc., New York}, 1999. \MR{1700749}

\bibitem{Birch:1971}
B.~J. Birch, \emph{Elliptic curves over {$Q$}: {{A}} progress report}, 1969
  {{Number Theory Institute}} ({{Proc}}. {{Sympos}}. {{Pure Math}}., {{Vol}}.
  {{XX}}, {{State Univ}}. {{New York}}, {{Stony Brook}}, {{N}}.{{Y}}., 1969),
  {Amer. Math. Soc., Providence, R.I.}, 1971, pp.~396--400. \MR{0314845}

\bibitem{BognerBrown:2015}
Christian Bogner and Francis Brown, \emph{Feynman integrals and iterated
  integrals on moduli spaces of curves of genus zero}, Communications in Number
  Theory and Physics \textbf{9} (2015), no.~1, 189--238. \MR{3339854}

\bibitem{Brown:2012}
Francis Brown, \emph{Mixed {{Tate}} motives over {$\mathbb Z$}}, Annals of
  Mathematics. Second Series \textbf{175} (2012), no.~2, 949--976. \MR{2993755}

\bibitem{Brown:2019c}
\bysame, \emph{From the {{Deligne-Ihara}} conjecture to multiple modular
  values}, Profinite Monodromy, {{Galois}} Representations, and {{Complex}}
  Functions (Masanobu Kaneko, ed.), vol. 2120, {RIMS K\^oky\^uroku}, 2019,
  p.~25.

\bibitem{Brown:2019b}
\bysame, \emph{A multi-variable version of the completed {{Riemann}} zeta
  function and other {$L$}-functions}, Profinite Monodromy, {{Galois}}
  Representations, and {{Complex}} Functions (Masanobu Kaneko, ed.), vol. 2120,
  {RIMS K\^oky\^uroku,}, 2019, p.~27.

\bibitem{BruggemanChoie:2016}
Roelof Bruggeman and Youngju Choie, \emph{Multiple period integrals and
  cohomology}, Algebra \& Number Theory \textbf{10} (2016), no.~3, 645--664.
  \MR{3513133}

\bibitem{BruggemanDiamantis:2016a}
Roelof Bruggeman and Nikolaos Diamantis, \emph{Fourier coefficients of
  {{Eisenstein}} series formed with modular symbols and their spectral
  decomposition}, J. Number Theory \textbf{167} (2016), 317--335. \MR{3504050}

\bibitem{Bump:1997a}
Daniel Bump, \emph{Automorphic forms and representations}, Cambridge
  {{Studies}} in {{Advanced Mathematics}}, vol.~55, {Cambridge University
  Press, Cambridge}, 1997. \MR{1431508}

\bibitem{Chen:1967}
Kuo-Tsai Chen, \emph{Iterated path integrals and generalized paths}, Bulletin
  of the American Mathematical Society \textbf{73} (1967), 935--938.
  \MR{217717}

\bibitem{Chen:1971}
\bysame, \emph{Algebras of iterated path integrals and fundamental groups},
  Transactions of the American Mathematical Society \textbf{156} (1971),
  359--379. \MR{275312}

\bibitem{Chen:1973a}
\bysame, \emph{Iterated {{Integrals}} of {{Differential Forms}} and {{Loop
  Space Homology}}}, Annals of Mathematics \textbf{97} (1973), no.~2, 217--246.

\bibitem{Chen:1977}
\bysame, \emph{Iterated path integrals}, Bulletin of the American Mathematical
  Society \textbf{83} (1977), no.~5, 831--879. \MR{454968}

\bibitem{ChintaDiamantisOSullivan:2002a}
Gautam Chinta, Nikolaos Diamantis, and Cormac O'Sullivan, \emph{Second order
  modular forms}, Acta Arith. \textbf{103} (2002), no.~3, 209--223.
  \MR{1905087}

\bibitem{ChintaHorozovOSullivan:2019}
Gautam Chinta, Ivan Horozov, and Cormac O'Sullivan, \emph{Noncommutative
  modular symbols and {{Eisenstein}} series}, Automorphic {{Forms}} and
  {{Related Topics}} ({Providence, Rhode Island}) (Samuele Anni, Jay Jorgenson,
  Lejla Smajlovi{\'c}, and Lynne Walling, eds.), Contemporary {{Mathematics}},
  vol. 732, {American Mathematical Society}, June 2019, pp.~27--45.

\bibitem{Choie:2016}
YoungJu Choie, \emph{Parabolic cohomology and multiple {{Hecke L-values}}}, The
  Ramanujan Journal \textbf{41} (2016), no.~1, 543--561.

\bibitem{ChoieDiamantis:2006}
YoungJu Choie and Nikolaos Diamantis, \emph{Rankin-{{Cohen}} brackets on higher
  order modular forms}, Multiple {{Dirichlet}} Series, Automorphic Forms, and
  Analytic Number Theory, Proc. {{Sympos}}. {{Pure Math}}., vol.~75, {Amer.
  Math. Soc., Providence, RI}, 2006, pp.~193--201. \MR{2279937}

\bibitem{ChoieIhara:2013}
YoungJu Choie and Kentaro Ihara, \emph{Iterated period integrals and multiple
  {{Hecke}} {$L$}-functions}, Manuscripta Mathematica \textbf{142} (2013),
  no.~1-2, 245--255. \MR{3081007}

\bibitem{ChoieMatsumoto:2016}
YoungJu Choie and Kohji Matsumoto, \emph{Functional equations for double series
  of {{Euler}} type with coefficients}, Advances in Mathematics \textbf{292}
  (2016), 529--557. \MR{3464029}

\bibitem{Constantinescu:2020a}
Petru Constantinescu, \emph{Distribution of modular symbols in
  {$\mathbb{H}^3$}}, International Mathematics Research Notices (2020),
  rnaa241.

\bibitem{Cremona:1997a}
John~E. Cremona, \emph{Algorithms for modular elliptic curves}, second ed.,
  {Cambridge University Press}, {Cambridge}, 1997. \MR{1628193 (99e:11068)}

\bibitem{DeitmarHorozov:2013}
Anton Deitmar and Ivan Horozov, \emph{Iterated {{Integrals}} and {{Higher Order
  Invariants}}}, Canadian Journal of Mathematics \textbf{65} (2013), no.~3,
  544--552.

\bibitem{DeligneGoncharov:2005}
Pierre Deligne and Alexander~B. Goncharov, \emph{{Groupes fondamentaux
  motiviques de Tate mixte}}, Annales scientifiques de l'\'Ecole Normale
  Sup\'erieure \textbf{38} (2005), no.~1, 1--56.

\bibitem{DiamantisSim:2008}
Nikolaos Diamantis and David Sim, \emph{The classification of higher-order cusp
  forms}, J. Reine Angew. Math. \textbf{622} (2008), 121--153. \MR{2433614}

\bibitem{DrappeauNordentoft:2022a}
Sary Drappeau and Asbj{\o}rn~Christian Nordentoft, \emph{Central values of
  additive twists of {{Maass}} forms {$L$}-functions}, arXiv:2208.14346, August
  2022.

\bibitem{DukeFriedlanderIwaniec:2002a}
William Duke, John~B. Friedlander, and Henryk Iwaniec, \emph{The subconvexity
  problem for {{Artin L-functions}}}, Invent. Math. \textbf{149} (2002), no.~3,
  489--577. \MR{1923476}

\bibitem{Fay:1977a}
John~D. Fay, \emph{Fourier coefficients of the resolvent for a {{Fuchsian}}
  group}, J. Reine Angew. Math. \textbf{293/294} (1977), 143--203. \MR{0506038}

\bibitem{FranconViennot:1979a}
Jean Fran{\c c}on and G{\'e}rard Viennot, \emph{Permutations selon leurs pics,
  creux, doubles mont\'ees et double descentes, nombres d'{{Euler}} et nombres
  de {{Genocchi}}}, Discrete Mathematics \textbf{28} (1979), no.~1, 21--35.
  \MR{542933}

\bibitem{GilFresan:2017}
Jos{\'e} Ignacio~Burgos Gil and Javier Fres{\'a}n, \emph{Multiple zeta values:
  From numbers to motives}, Clay Mathematics Proceedings, to appear (2017),
  379.

\bibitem{Goldfeld:1999b}
Dorian Goldfeld, \emph{The distribution of modular symbols}, Number Theory in
  Progress, {{Vol}}. 2 ({{Zakopane-Ko\'scielisko}}, 1997), {de Gruyter},
  {Berlin}, 1999, pp.~849--865. \MR{2000g:11040}

\bibitem{Goldfeld:1999a}
\bysame, \emph{Zeta functions formed with modular symbols}, Automorphic Forms,
  Automorphic Representations, and Arithmetic ({{Fort Worth}}, {{TX}}, 1996),
  Proc. {{Sympos}}. {{Pure Math}}., vol.~66, {Amer. Math. Soc.}, {Providence,
  RI}, 1999, pp.~111--121. \MR{2000g:11039}

\bibitem{Hain:1998}
Richards Hain, \emph{The {{Hodge}} de {{Rham}} theory of relative {{Malcev}}
  completion}, Annales Scientifiques de l'\'Ecole Normale Sup\'erieure
  \textbf{31} (1998), no.~1, 47--92.

\bibitem{IharaTakamuki:2001}
Kentaro Ihara and Takashi Takamuki, \emph{The quantum {${\mathfrak g}_2$}
  invariant and relations of multiple zeta values}, Journal of Knot Theory and
  Its Ramifications \textbf{10} (2001), no.~07, 983--997.

\bibitem{Iwaniec:2002a}
Henryk Iwaniec, \emph{Spectral methods of automorphic forms}, second ed.,
  Graduate {{Studies}} in {{Mathematics}}, vol.~53, {American Mathematical
  Society}, {Providence, RI}, 2002. \MR{1942691}

\bibitem{IwaniecKowalski:2004a}
Henryk Iwaniec and Emmanuel Kowalski, \emph{Analytic number theory}, American
  {{Mathematical Society Colloquium Publications}}, vol.~53, {American
  Mathematical Society}, {Providence, RI}, 2004. \MR{2061214}

\bibitem{JorgensonOSullivan:2008a}
Jay Jorgenson and Cormac O'Sullivan, \emph{Unipotent vector bundles and
  higher-order non-holomorphic {{Eisenstein}} series}, J. Th\'eor. Nombres
  Bordeaux \textbf{20} (2008), no.~1, 131--163. \MR{2434161}

\bibitem{KlebanZagier:2003a}
Peter Kleban and Don Zagier, \emph{Crossing probabilities and modular forms},
  J. Statist. Phys. \textbf{113} (2003), no.~3-4, 431--454. \MR{2013692}

\bibitem{KleiberStoyanov:2013}
Christian Kleiber and Jordan Stoyanov, \emph{Multivariate distributions and the
  moment problem}, Journal of Multivariate Analysis \textbf{113} (2013), 7--18.

\bibitem{Kontsevich:1999}
Maxim Kontsevich, \emph{Operads and {{Motives}} in {{Deformation
  Quantization}}}, Letters in Mathematical Physics \textbf{48} (1999), no.~1,
  35--72.

\bibitem{Kulkarni:1991}
Ravi~S. Kulkarni, \emph{An {{Arithmetic-Geometric Method}} in the {{Study}} of
  the {{Subgroups}} of the {{Modular Group}}}, American Journal of Mathematics
  \textbf{113} (1991), no.~6, 1053.

\bibitem{KurthLong:2008}
Chris~A. Kurth and Ling Long, \emph{Computations with finite index subgroups of
  {$PSL_2(\mathbb Z)$} using {{Farey Symbols}}}, Advances in {{Algebra}} and
  {{Combinatorics}}, {World Scientific}, June 2008, pp.~225--242.

\bibitem{LeMurakami:1996}
Thang Tu~Quoc Le and Jun Murakami, \emph{Kontsevich's integral for the
  {{Kauffman}} polynomial}, Nagoya Mathematical Journal \textbf{142} (1996),
  no.~none, 39--65.

\bibitem{LeMurakami:1995}
Tu~Quoc~Thang Le and Jun Murakami, \emph{Kontsevich's integral for the
  {{Homfly}} polynomial and relations between values of multiple zeta
  functions}, Topology and its Applications \textbf{62} (1995), no.~2,
  193--206.

\bibitem{LeeSun:2019}
Jungwon Lee and Hae-Sang Sun, \emph{Dynamics of continued fractions and
  distribution of modular symbols}, arXiv:1902.06277 (2019), 42 pp.

\bibitem{Loeve:1977a}
Michel Lo{\`e}ve, \emph{Probability theory. {{I}}}, fourth ed., Graduate
  {{Texts}} in {{Mathematics}}, vol.~45, {Springer-Verlag}, {New York}, 1977.
  \MR{0651017}

\bibitem{Maass:1949a}
Hans Maass, \emph{{{Über}} eine neue {{Art}} von nichtanalytischen automorphen
  {{Funktionen}} und die {{Bestimmung Dirichletscher Reihen}} durch
  {{Funktionalgleichungen}}}, Math. Ann. \textbf{121} (1949), 141--183.
  \MR{0031519}

\bibitem{Manin:1972a}
Ju.~I. Manin, \emph{Parabolic points and zeta functions of modular curves},
  Izv. Akad. Nauk SSSR Ser. Mat. \textbf{36} (1972), 19--66. \MR{0314846}

\bibitem{Manin:2005}
Yuri~I. Manin, \emph{Iterated {{Shimura}} integrals}, Mosc. Math. J. \textbf{5}
  (2005), no.~4, 869--881, 973. \MR{2266463}

\bibitem{Manin:2006}
\bysame, \emph{Iterated integrals of modular forms and noncommutative modular
  symbols}, Algebraic Geometry and Number Theory, Progr. {{Math}}., vol. 253,
  {Birkh\"auser Boston}, {Boston, MA}, 2006, pp.~565--597. \MR{2263200}

\bibitem{Manin:2009}
\bysame, \emph{Lectures on modular symbols}, Arithmetic Geometry, Clay
  {{Math}}. {{Proc}}., vol.~8, {Amer. Math. Soc.}, {Providence, RI}, 2009,
  pp.~137--152. \MR{2498060}

\bibitem{Mazur:1973a}
Barry Mazur, \emph{Courbes elliptiques et symboles modulaires},  (1973),
  277--294. Lecture Notes in Math., Vol. 317. \MR{0429921}

\bibitem{MazurRubin:2021}
Barry Mazur and Karl Rubin, \emph{Arithmetic {{Conjectures Suggested}} by the
  {{Statistical Behavior}} of {{Modular Symbols}}}, Experimental Mathematics
  \textbf{0} (2021), no.~0, 1--16.

\bibitem{Michel:2007}
Philippe Michel, \emph{Analytic number theory and families of automorphic
  {{L-functions}}}, Automorphic Forms and Applications, {{IAS}}/{{Park City
  Math}}. {{Ser}}., vol.~12, {Amer. Math. Soc.}, {Providence, RI}, 2007,
  pp.~181--295. \MR{2331346}

\bibitem{Nadarajah:2003}
Saralees Nadarajah, \emph{The {{Kotz-type}} distribution with applications},
  Statistics \textbf{37} (2003), no.~4, 341--358.

\bibitem{Nordentoft:2018}
Asbj{\o}rn~Christian Nordentoft, \emph{Central values of additive twists of
  modular {$L$}-functions}, arXiv:1812.08378 (2018), 1--40.

\bibitem{Nordentoft:2021c}
\bysame, \emph{Central values of additive twists of cuspidal {$L$}-functions},
  Journal f\"ur die Reine und Angewandte Mathematik. \textbf{776} (2021),
  255--293. \MR{4279108}

\bibitem{OSullivan:2000a}
Cormac O'Sullivan, \emph{Properties of {{Eisenstein}} series formed with
  modular symbols}, J. Reine Angew. Math. \textbf{518} (2000), 163--186.
  \MR{2000j:11073}

\bibitem{Petridis:2002a}
Yiannis~N. Petridis, \emph{Spectral deformations and {{Eisenstein}} series
  associated with modular symbols}, Int. Math. Res. Not. (2002), no.~19,
  991--1006. \MR{1 903 327}

\bibitem{PetridisRisager:2004a}
Yiannis~N. Petridis and Morten~S. Risager, \emph{Modular symbols have a normal
  distribution}, Geom. Funct. Anal. \textbf{14} (2004), no.~5, 1013--1043.
  \MR{2105951}

\bibitem{PetridisRisager:2018a}
\bysame, \emph{Arithmetic statistics of modular symbols}, Inventiones
  mathematicae \textbf{212} (2018), no.~3, 997--1053. \MR{3802302}

\bibitem{Reutenauer:1993}
Christophe Reutenauer, \emph{Free {{Lie}} algebras}, London {{Mathematical
  Society Monographs}}. {{New Series}}, vol.~7, {The Clarendon Press, Oxford
  University Press, New York}, 1993. \MR{1231799}

\bibitem{Roelcke:1966a}
Walter Roelcke, \emph{Das {{Eigenwertproblem}} der automorphen {{Formen}} in
  der hyperbolischen {{Ebene}}. {{I}}, {{II}}.},  (1966).

\bibitem{Selberg:1956}
Atle Selberg, \emph{Harmonic analysis and discontinuous groups in weakly
  symmetric {{Riemannian}} spaces with applications to {{Dirichlet}} series},
  J. Indian Math. Soc. (N.S.) \textbf{20} (1956), 47--87. \MR{0088511}

\bibitem{Selberg:1965a}
\bysame, \emph{On the estimation of {{Fourier}} coefficients of modular forms},
  Proc. {{Sympos}}. {{Pure Math}}., {{Vol}}. {{VIII}}, {Amer. Math. Soc.},
  {Providence, R.I.}, 1965, pp.~1--15. \MR{0182610}

\bibitem{Waldschmidt:2002}
Michel Waldschmidt, \emph{Multiple polylogarithms: An introduction}, Number
  Theory and Discrete Mathematics ({{Chandigarh}}, 2000), Trends {{Math}}.,
  {Birkh\"auser, Basel}, 2002, pp.~1--12. \MR{1952273}

\bibitem{Young:2019}
Matthew~P. Young, \emph{Explicit calculations with {{Eisenstein}} series},
  Journal of Number Theory \textbf{199} (2019), 1--48.

\end{thebibliography}
\providecommand{\bysame}{\leavevmode\hbox to3em{\hrulefill}\thinspace}
\providecommand{\MR}{\relax\ifhmode\unskip\space\fi MR }
\providecommand{\MRhref}[2]{%
  \href{http://www.ams.org/mathscinet-getitem?mr=#1}{#2}
}
\providecommand{\href}[2]{#2}

\end{document}